\documentclass[lewno11pt,smallextended,envcountsect,envcountsame]{svjour3} 
\usepackage[margin=1.0in]{geometry} 
\usepackage{graphicx}
\usepackage{mathrsfs} \usepackage{fix-cm} \usepackage{amsmath} \usepackage{amssymb} \usepackage{latexsym} \usepackage{dsfont} \usepackage{xcolor} \usepackage{cite} \usepackage{dsfont}
\usepackage{enumitem}  \def\disp{\displaystyle}     
\def\Limsup{\mathop{{\rm Lim}\,{\rm sup}}}    \def\tto{\;{\lower 1pt
\hbox{$\rightarrow$}}\kern -10pt \hbox{\raise 2pt \hbox{$\rightarrow$}}\;}  \def\Hat{\widehat}  \def\Bar{\overline} \def\ra{\rangle}
\def\la{\langle} \def\ve{\varepsilon} \def\B{\mathbb{B}} \def\h{\hfill\Box} \def\R{\mathbb{R}} \def\N{\mathbb{N}}  \def\ox{\bar{x}} \def\oy{\bar{y}} \def\ooy{\bar{\y}}
         
 \def\vt{\vartheta} \def\co{\mbox{\rm co}\,}
\def\gph{\operatorname{gph}} \def\epi{\mbox{\rm epi}\,}   \def\dom{\mbox{\rm
dom}\,}      
 \def\dist{\mbox{\rm dist}\,}    \def\inter{\mbox{\rm int}\,} 
   
\def\h{\hfill\square} \def\dn{\downarrow}   \def\ph{\varphi} \def\emp{\emptyset} \def\st{\stackrel} \def\oR{\Bar{\R}} \def\N{\mathbb{N}} 
 \def\gg{\gamma} \def\tt{\tilde t} \def\dd{\delta} \def\al{\alpha}     
 \newcounter{count}  
\DeclareMathOperator{\subm}{\partial} \DeclareMathOperator{\subf}{\Hat{\partial}}   
   \newcommand{\Intf}[1]{\mathrm{E}_{#1}} \newcommand{\Intfset}[1]{\mathrm{E}_{#1}}
 \newcommand{\der}[3]{ d{#1}(#2)(#3)}  \DeclareMathOperator*{\esssup}{ess\,sup} 
\DeclareMathOperator{\1}{\mathds{1}} \newcommand{\T}{T} \def\sce{\setcounter{equation}{0}} \newcommand{\Rex}{\overline{\mathbb{R}}} \renewcommand{\theequation}{{\thesection}.\arabic{equation}}
\let\epsilon\varepsilon \DeclareMathAlphabet{\mathpzc}{OT1}{pzc}{m}{it} 

\usepackage{mathalfa}   \def\u{\mathpzc{u}} \def\v{\mathpzc{v}} \def\w{\mathpzc{w}} \def\x{\mathpzc{x}} \def\y{\mathpzc{y}} \def\z{\mathpzc{z}}

\def\X{{\R^n} }
\def\Y{\R^{m} }
\def\Z{\R^{q}}
\def\Leb{\textnormal{L} }

\begin{document}
\renewcommand{\theequation}{{\thesection}.\arabic{equation}}

\title{ Generalized Leibniz rules and Lipschitzian stability for expected-integral mappings\thanks{Research of the first author was partially supported by the USA National Science Foundation
under grants DMS-1512846 and DMS-1808978, by the USA Air Force Office of Scientific Research under grant \#15RT04, and by the Australian Research Council under Discovery Project DP-190100555.
Research of the second author was partially supported by ANID grant: Fondecyt Regular 1200283 and Fondecyt Regular 1190110.}} \author{Boris S. Mordukhovich \and \mbox{Pedro P\'erez-Aros}}
\institute{Boris S. Mordukhovich \at Department of Mathematics, Wayne State University, Detroit, Michigan 48202, USA\\ \email{boris@math.wayne.edu}\\ \and Pedro P\'erez-Aros \at Instituto de
Ciencias de la Ingenier\'ia, Universidad de O'Higgins, Rancagua, Chile\\ \email{pedro.perez@uoh.com}}

\date{}

\maketitle

\begin{abstract}   This paper is devoted to the study of the expected-integral multifunctions given in the form \begin{equation*} \Intfset{\Phi}(x):=\int_T\Phi_t(x)d\mu, \end{equation*} where
$\Phi\colon T\times\X\tto\Y$ is a set-valued mapping on a measure space $(T,\mathcal{A},\mu)$. Such multifunctions appear in applications to stochastic programming, which require developing
efficient calculus rules of generalized differentiation. Major calculus rules are developed in this paper for coderivatives of multifunctions $\Intfset{\Phi}$ and second-order subdifferentials of
the corresponding expected-integral functionals with applications to constraint systems arising in stochastic programming. The paper is self-contained with presenting in the preliminaries some
needed results on sequential first-order subdifferential calculus of expected-integral functionals taken from the first paper of this series.\vspace*{-0.05in} \keywords{  \and Stochastic
programming \and Generalized differentiation \and Integral multifunctions \and Leibniz rules  \and Lipschitzian stability}\vspace*{-0.05in}

\subclass{Primary: 49J53, 90C15, 90C34 \and Secondary: 49J52}\vspace*{-0.05in}
\end{abstract}

\section{Introduction}\label{intro}\sce\vspace*{-0.1in}

{\em Stochastic programming} has been highly recognized as an important area of optimization theory with a variety of practical applications; see, e.g., the book \cite{sdr} and the references
therein. Although advanced methods of variational analysis and generalized differentiation have been used in the study and applications of stochastic programming (see, e.g.,
\cite{ah,ap,burke,chp20,dent-rusz,hhp,hr,mp20,mor-sag18,mor-sag19} among other publications), there is no comparison between the number of currently achieved results in this vein for stochastic
problems and the broad implementation of the aforementioned methods in deterministic optimization.

The primal motivation for our study is to narrow the gap between deterministic and stochastic applications of {\em variational analysis} and {\em generalized differentiation} to optimization and related
problems. The underlying feature of stochastic problems is the presence of {\em integration} with respect to probability measures over extended-real-valued variable integrands as well as over
vector-valued and set-valued mappings. Then efficient {\em calculus rules} of their generalized differentiation are in strong demand.

In our preceding paper \cite{mp20a} we established various {\em sequential} versions of the generalized {\em Leibniz rule} (subdifferentiation under the integral sign) in terms of {\em regular}
subgradients of the {\em expected-integral functionals} that are defined by \begin{equation}\label{eif} \Intfset{\varphi}(x,\y):=\int_T\varphi_t\big(x,\y(t)\big)d\mu, \end{equation} where
$x\in\R^n$, $\y\in\Leb^1(T;\R^m)$, $\ph_t(x,y):=\ph(t,x,y)$, and $\ph\colon T\times\R^n\times\R^m\to\Rex:=[-\infty,\infty]$ is an extended-real-valued function on a {\em complete finite measure
space} $(T,\mathcal{A},\mu)$, which is assumed below.

The goal of this paper is to proceed much further and to study {\em expected-integral multifunctions} (set-valued mapping) given in the form \begin{equation}\label{def:set-valued:Exp}
\Intfset{\Phi}(x):=\int_{T}\Phi_t(x)d\mu, \end{equation} where $\Phi\colon T\times\X\tto\Y$ is a set-valued (in particular, single-valued $\Phi\colon T\times\X\to\Y$) mapping defined on a measure
space. The reader may consult the survey paper \cite{hess} for an overview on the integration of random set-valued mappings and set-valued probability theory. We also refer the reader to the
recent publications \cite{burke,dent-rusz} and the bibliographies therein for new developments and applications. Our current work presented in this paper is mostly theoretical, while we'll
discuss the aimed stochastic applications in the concluding section below.

A natural extension of subdifferentials to the case of (single-valued and set-valued) mappings is provided by {\em coderivatives}, and thus we focus here on deriving coderivative versions of
Leibniz's rule for expected-integral multifunctions of type \eqref{def:set-valued:Exp}. To the best of our knowledge, this has never been done in the literature. In contrast to \cite{mp20a}, our
major results are obtained in the {\em exact/pointwise} form, i.e., they are formulated exactly at the points in question via the {\em limiting} constructions. Such results are clearly more
convenient for the theory and applications than sequential/fuzzy ones and provide new rules even for extended-integral functionals \eqref{eif} in comparison with \cite{mp20a}.

An important role in deriving pointwise {\em coderivative Leibniz-type rules} is played by the new {\em integrable quasi-Lipschitz property} for set-valued random {\em normal integrands}
$\Phi_t(x)$ in \eqref{def:set-valued:Exp}. For deterministic multifunctions, the introduced property is equivalent to the (Aubin) {\em Lipschitz-like} property due to the (Mordukhovich) {\em
coderivative criterion} of variational analysis (see \cite{m93,m06,rw}), while these two Lipschitzian properties are essentially different in stochastic frameworks. It happens, in particular, for
spaces with nonatomic measures, where the integrable quasi-Lipschitz property occurs to be equivalent to the integrable counterpart of the (Hausdorff) {\em locally Lipschitzian} property for
random multifunctions. All of this leads us to pleasable conclusions about {\em Lipschitz stability} of random {\em feasible solution mappings} in problems of stochastic programming. They
include, in particular, efficient conditions on {\em quantitative continuity} of parametric sets of feasible solutions to stochastic programs with inequality constraints.

Along with the study of expected-integral multifunctions \eqref{def:set-valued:Exp} defined by arbitrary integrand mappings $\Phi_t(x)$, we consider {\em structural} multifunctions of the type
\begin{equation}\label{composition}
\Phi_t(x)=F\big(t,g_t(x)\big)\;\mbox{ for all }\;x\in U\;\mbox{ and a.e.\ }\;t\in T
\end{equation}
defined as compositions of set-valued mappings $F$ and single-valued mappings $g$ between finite-dimensional spaces. It is shown that such multifunctions \eqref{def:set-valued:Exp} exhibit integrable
Lipschitz stability in the case of convex outer mappings $F$ and smooth inner mappings $g$ in \eqref{composition}. Then we introduce a new class of {\em integrable amenable} set-valued compositions and
use them to establish {\em chain rules} of {\em coderivative integration} for the corresponding expected-integral multifunctions. The obtained results are specified for the case where set-valued
mappings $F$ in \eqref{composition} are given by inequality constraints, which is a typical case for feasible solutions mappings in constrained stochastic programming.

Finally, in this paper we establish, for the first time in the literature, Leibniz-type rules for {\em second-order subdifferentials} of expected-integral functionals given in the form
\begin{equation}\label{eif:expected}
\Intfset{\varphi}(x):=\int_T\varphi_t\big(x\big)d\mu,
\end{equation}
generated by extended-real-valued functions $\ph\colon T\times\X\to\Rex$. The obtained results are specified for the case where $\ph$ in \eqref{eif:expected} is given as a {\em maximum function}, which is
important for applications to problems of stochastic programming.\vspace*{0.03in}

The rest of the paper is organized as follows. Section~\ref{sec2} recalls and discusses the basic constructions of {\em generalized differentiation} in variational analysis that are systematically
employed in the subsequent material. In Section~\ref{sec2a} we briefly review some notions of measurability and integration for multifunctions and then present the {\em sequential Leibniz rules}
for subdifferentiation of expected-integral functionals \eqref{eif} taken from \cite{mp20} and used in what follows.

Section~\ref{sec3} is devoted to the study of {\em Lipschitzian properties} of random multifunctions and presents various characterizations of all the three properties mentioned above with establishing
relationships between them in general measure spaces as well as in spaces with purely atomic and nonatomic measures.

In Section~\ref{sec4} we obtain several results of the new type labeled as {\em coderivative Leibniz rules}, which evaluate both {\em regular} and {\em limiting coderivatives} of
expected-integral multifunctions \eqref{def:set-valued:Exp} via the integration of random integrands $\Phi_t$ therein. The results obtained include {\em sequential} Leibniz rules for regular
coderivatives and {\em pointwise} ones for the limiting coderivative construction. As a consequence of the coderivative Leibniz rules of the latter type and the aforementioned coderivative
characterizations of Lipschitzian properties, we establish efficient conditions for Lipschitz stability of the expected-integral multifunctions in terms of coderivatives of their integrands
$\Phi_t(x)$.

Section~\ref{sec5} addresses the setting where the integrand $\Phi_t(x)$ is given in the {\em composition} form \eqref{composition} defined by integrable amenable mappings. The obtained conditions for
Lipschitz stability and {\em composite Leibniz rules} are specified here for random sets of {\em feasible solutions} in constrained stochastic programming.

Considering in Section~\ref{sec6} expected  functionals \eqref{eif:expected}, we derive sequential and pointwise {\em second-order Leibniz rules} in terms of second-order subdifferentials of two types,
which have been well known in variational analysis while have never been used in the study of expected-integral functionals and applications to stochastic programming. The concluding Section~\ref{sec7}
summarizes the main achievements of this paper and discusses some directions of our future research and applications.\vspace*{0.03in}

In this paper we use the standard notation from variational analysis, generalized differentiation, and stochastic programming; see, e.g., \cite{m06,rw,sdr}. Recall that the extended real line is denoted
by $\Rex:=[-\infty,\infty]$ with the conventions that $(\pm\infty)\cdot 0=0\cdot(\pm\infty)=0$ and $\infty-\infty=-\infty+\infty=\infty$. The symbol $\mathbb{B}_r(x)$ stands for the closed ball centered
at $x$ with radius $r>0$, while the unit closed ball of the space in question is denoted simply by $\mathbb{B}$. Given a nonempty set $\Omega\subset\X$, its {\em indicator function} $\dd_\Omega$ is
defined by $\delta_\Omega(x):=0$ if $x\in\Omega$ and $\delta_\Omega(x):=\infty$ otherwise, while the {\em characteristic function} $\1_\Omega$ is defined by $\1_\Omega(x):=1$ for $x\in\Omega$ and
$\1_\Omega(x):=0$ for $x\notin\Omega$. The symbol $x\st{\Omega}{\to}\ox$ means that $x\to\ox$ with $x\in\Omega$, and $\Omega^c$ denotes the complement of $\Omega$. Finally, $\N:=\{1,2,\ldots\}$ and
$\R_+:=\{\al\in\R\;|\;\al> 0\}$.\vspace*{-0.2in}

\section{Preliminaries from Generalized Differentiation}\label{sec2}\sce\vspace*{-0.1in}

In this section we recall some basic constructions of {\em generalized differentiation} in variational analysis that are broadly used in what follows.

Let $\Phi\colon\X\tto\Y$ be a set-valued mapping with the {\em domain} and {\em graph} given by, respectively,
\begin{equation*}
\dom\Phi:=\big\{x\in\R^n\;\big|\;\Phi(x)\ne\emp\big\}\;\mbox{ and }\;\gph\Phi:=\big\{(x,y)\in\R^n\times\R^m\;\big|\;y\in\Phi(x)\big\}.
\end{equation*}
The {\em Painlev\'e-Kuratowski outer limit} of $F$ as $x\to\ox$ is defined by
\begin{equation}\label{pk}
\Limsup_{x\to\ox}\Phi(x):=\big\{v\in\R^m\big|\;\exists\,\mbox{ seqs. }\;x_k\to\ox,\;v_k\to v\;\mbox{ s.t. }\;v_k\in\Phi(x_k)\big\}.
\end{equation}
The {\em regular/Fr\'echet normal cone} to $\Omega$ at $\ox\in\Omega$ is
\begin{equation}\label{rnc}
\Hat N(\ox;\Omega):=\Big\{x^*\in\R^n\;\Big|\;\limsup_{x\st{\Omega}{\to}\ox}\frac{\la x^*,x-\ox\ra}{\|x-\ox\|}\le 0\Big\}
\end{equation}
with $\Hat N(\ox;\Omega):=\emp$ if $\ox\notin\Omega$. The {\em limiting/Mordukhovich normal cone} to $\Omega$ at $\ox$ is defined via \eqref{pk} by
\begin{equation}\label{lnc}
N(\ox;\Omega):=\Limsup_{x\to\ox}\Hat N(x;\Omega).
\end{equation}
Note that, in contrast to \eqref{rnc}, the limiting normal cone \eqref{lnc} and the associated coderivative and subdifferential constructions (see below) are {\em robust}, meaning that they are
closed-graph multifunctions with respect to perturbations of the initial points.

Based on the normal cones \eqref{rnc} and \eqref{lnc} to the graph of $\Phi\colon\X\tto\Y$ at $(\ox,\oy)\in\gph\Phi$, we define the corresponding {\em regular} and {\em basic/limiting
coderivatives} of $\Phi$ at $(\ox,\oy)$ for all $y^*\in\Y$ by, respectively, \begin{equation}\label{rcod} \Hat D^*\Phi(\ox,\oy)(y^*):=\big\{x^*\in\X\;\big|\;(x^*,-y^*)\in\Hat
N\big((\ox,\oy);\gph\Phi\big)\big\}, \end{equation} \begin{equation}\label{lcod} D^*\Phi(\ox,\oy)(y^*):=\big\{x^*\in\X\;\big|\;(x^*,-y^*)\in N\big((\ox,\oy);\gph\Phi\big)\big\}, \end{equation}
where we omit $\oy$ in the coderivative notation of $\Phi$ if it is a singleton $\{\Phi(\ox)\}$. If $\Phi\colon\X\to\Y$ is ${\cal C}^1$-smooth around $\ox$ (in fact, merely strictly
differentiable at this point), then both coderivatives above reduce to the {\em adjoint} (transpose) Jacobian matrix linearly applied to any $y^*\in\Y$: \begin{equation}\label{cod-smooth} \Hat
D^*\Phi(\ox)(y^*)=D^*\Phi(\ox)(y^*)=\big\{\nabla\Phi(\ox)^*y^*\big\}. \end{equation} In general, both coderivatives \eqref{rcod} and \eqref{lcod} are positively homogeneous multifunctions, where
\eqref{rcod} is always convex-valued, while \eqref{lcod} is not even for very simple convex functions as, e.g., for $\Phi(x):=|x|$ at $\ox=0\in\R$. Nevertheless, the limiting coderivative
\eqref{lcod}, together with the normal cone \eqref{lnc} and the associated first-order and second-order subdifferentials of extended-real-valued functions presented below, enjoy {\em full
pointwise calculus} based on {\em variational} and {\em extremal principles} of variational analysis; see the books \cite{m06,m18,rw} for more details and references. Unfortunately, it is not the
case for the corresponding regular constructions, for which only ``fuzzy" results are available. On the other hand, it is convenient to have the following representation:
\begin{equation}\label{Lim_repr_cod} D^*\Phi(\ox,\oy)(y^\ast)=\Limsup_{(x,y)\st{\gph\Phi}{\to}(\ox,\oy),\,v^*\to y^*}\Hat{D}^*\Phi(x,y)(v^*) \end{equation} of \eqref{lcod} at $(\ox,\oy)$ as the
outer limit of \eqref{rcod} at points nearby.

Next we consider an extended-real-valued function $\ph\colon\R^n\to\oR$ with its {\em domain} and {\em epigraph} that are defined, respectively, by
\begin{equation*}
\dom\ph:=\big\{x\in\R^n\;\big|\;\ph(x)<\infty\big\}\;\mbox{ and }\;\epi\ph:=\big\{(x,\al)\in\R^{n+1}\;\big|\;\al\ge\ph(x)\big\}.
\end{equation*}
The {\em properness} of $\ph$, which we assume from now on, means that $\dom\ph\ne\emp$ and $\ph(x)>-\infty$ for all $x\in\X$. Applying the normal cones \eqref{rnc} and \eqref{lnc} to the epigraph of
$\ph$ at $(\ox,\ph(\ox))$ with $|\ph(\ox)|\ne\infty$ gives us the corresponding (geometric) definitions of the {\em regular} and {\em limiting/basic subdifferentials}
\begin{equation}\label{rsub}
\Hat\partial\ph(\ox):=\big\{x^*\in\R^n\;\big|\;(x^*,-1)\in\Hat N\big((\ox,\ph(\ox));\epi\ph\big)\big\},
\end{equation}
\begin{equation}\label{lsub}
\partial\ph(\ox):=\big\{x^*\in\R^n\;\big|\;(x^*,-1)\in N\big((\ox,\ph(\ox));\epi\ph\big)\big\},
\end{equation}
while the reader is referred to \cite{m06,m18,rw} for equivalent analytic descriptions, various properties, and applications. We say that the function $\ph$ is {\em lower regular} at $\ox$ if $\ph$ is
finite and $\subf\ph(\ox)=\subm\ph(\ox)$.

In the finite-dimensional setting under consideration, the regular subdifferential \eqref{rsub} admits the following useful representation:
\begin{equation}\label{rep_reg_sub}
\Hat\partial\ph(\ox)=\big\{x^\ast\in\X\;\big|\;\langle x^\ast,w\rangle\le\der{\ph}{x}{w}\;\text{ for all }\;w\in\X\big\}
\end{equation}
in terms of the (Dini-Hadamard) {\em subderivative} of $\ph$ at $\ox$ with respect to the direction $w$ defined by
\begin{equation}\label{subderivative}
\der{\ph}{\ox}{w}=\liminf\limits_{t\dn 0,\,u\to w}\frac{\ph(\ox+tu)-\ph(x)}{t}.
\end{equation}

Finally in this section, we recall two notions of {\em second-order subdifferentials} of extended-real-valued functions that are obtained by the scheme of \cite{m92} as coderivatives of subgradient
mappings; see \cite{m06,m18} for more details. Given $\ph\colon\X\to\Rex$ finite at $\ox$ and $\ox^\ast\in\partial\ph(\ox)$, the {\em basic second-order subdifferential} of $\ph$ at $\ox$ relative to
$\ox^\ast$ is defined as the set-valued mapping $\partial^2\ph(\ox,\ox^\ast)\colon\X\tto\X$ with the values
\begin{equation}\label{2nd}
\partial^2\ph(\ox,\ox^\ast)(v^\ast):=\big(D^\ast\partial\ph\big)(\ox,\ox^\ast)(v^\ast)\;\mbox{ whenever }\;v^\ast\in\X.
\end{equation}
The {\em combined second-order subdifferential} of $\ph$ at $\ox$ with respect to $\ox^\ast$ is defined similarly to \eqref{2nd} by replacing the basic coderivative \eqref{lcod} in \eqref{2nd} with the
regular one \eqref{rcod} as
\begin{equation}\label{2nd1}
\breve{\partial}^2\ph(\ox,\ox^\ast)(v^\ast):=\big(\Hat D^\ast\partial\ph\big)(\ox,\ox^\ast)(v^\ast)\;\mbox{ for all }\;v^\ast\in\X.
\end{equation}
The indication of $\ox^*$ is dropped in the notation of \eqref{2nd} and \eqref{2nd1} when $\partial\ph(\ox)=\{\nabla\ph(\ox)\}$. Note that for ${\cal C}^2$-smooth functions $\ph$ we have
\begin{equation*}
\partial^2\ph(\ox)(v^\ast)=\breve{\partial}^2\ph(\ox)(v^\ast)=\big\{\nabla^2\ph(\ox)v^*\big\}\;\mbox{ for any }\;v^*\in\X
\end{equation*}
via the (symmetric) Hessian matrix. Calculus rules for the second-order subdifferentials and their computations for remarkable classes of functions can be found in \cite{m06,m18,mr} and the references
therein.\vspace*{-0.2in}

\section{Subdifferentiation of Expected-Integral Functionals}\label{sec2a}\sce\vspace*{-0.1in}

In this section we review the needed notions of measurability and integration for set-valued mappings and then present some results on sequential Leibniz rules for expected-integral functionals obtained
in \cite{mp20}.

Throughout the paper, $(T,\mathcal{A},\mu)$ is a complete finite measure space as mentioned in Section~\ref{intro}. To avoid confusions, we use the special font (as, e.g., $\v,\w,\x,\y,\z$, etc.) to
denote vector functions defined on $T$. When $p\in[1,\infty]$, the notation $\Leb^p({\T},\R^n)$ stands for the space of all the (equivalence classes by the relation equal almost everywhere) measurable
functions $\x$ such that the scalar function $\|\x(\cdot)\|^p$ is integrable for $p\in[1,\infty)$ and essentially bounded for $p=\infty$. The norm in $\Leb^p(T,\mathbb{R}^n)$ is denoted by $\|\cdot\|_p$,
and the points in $\X$ are identified with constant functions in $\Leb^p(T,\X)$. Thus for $x\in\X$ and $\x\in\Leb^p(T,\X)$ we have the expressions
\begin{align*}
\|x-\x\|_p &:=\left(\int_T\|x-\x(t)\|^p d\mu\right)^{1/p}\;\mbox{ as }\;p\in[1,\infty),\\
\|x-\x\|_\infty&:=\esssup\limits_{t\in T}\|x-\x(t)\|.
\end{align*}

Recall that a set-valued mapping $F\colon T\tto\R^n$ is {\em measurable} if for every open set $U\subset\mathbb{R}^n$ the inverse image $F^{-1}(U):=\{t\in T\;|\;F(t)\cap U\ne\emp\}$ is
measurable, i.e., $F^{-1}(U)\in\mathcal{A}$. The mapping $F$ is said to be {\em graph measurable} if $\gph F\in\mathcal{A}\otimes\mathcal{B}(\mathbb{R}^n)$, where $\mathcal{B}(\mathbb{R}^n)$ is
the Borel $\sigma$-algebra, i.e., the $\sigma$-algebra generated by open subsets of $\mathbb{R}^n$. It is easy to see, due to the completeness of the measure space $(T,\mathcal{A},\mu)$, that a
multifunction $F$ with closed values is measurable if and only if it is graph measurable. The {\em Aumann integral} of $F\colon T\tto\R^n$ over a measurable set $S\in\mathcal{A}$ is defined by
\begin{equation}\label{aum} \int_S F(t)d\mu:=\bigg\{\int_S\x^*(t)d\mu\;\bigg|\;\x^*\in{\Leb}^1(\T,\X)\textnormal{ and }\x^*(t)\in F(t)\text{ a.e.}\bigg\}. \end{equation} The fundamental {\em
Lyapunov convexity theorem} says that if the measure $\mu$ is {\em nonatomic} on $T$ (i.e., there is no $A\in{\cal A}$ such that $\mu(A)>0$ and for any $B\subset A$ with $B\in {\cal A}$ and
$\mu(B)<\mu(A)$ it follows that $\mu(B)=0$), then the integral set in \eqref{aum} is closed and convex in $\R^n$ provided that the multifunction $F$ is measurable and uniformly bounded by a
summable on $T$ function. The reader is referred to the book \cite{du} and its bibliography for the Lyapunov convexity theorem and its infinite-dimensional extensions; see also \cite{mor-sag18}
for the most recent results in this direction.

Together with \eqref{aum}, in this paper we often use the following notion for extended-real-valued functions of two variables. A function $\ph\colon T\times\mathbb{R}^n\to\Rex$ is called a {\em
normal integrand} if the multifunction $t\to\epi\ph_t$ is measurable with closed values. By the completeness of $(T,\mathcal{A},\mu)$, this is equivalent to saying that $\ph$ is

$\mathcal{A}\otimes\mathcal{B}(\mathbb{R}^n)$-measurable, and that for every $t\in T$ the function $\ph_t:=\ph(t,\cdot)$ is lower semicontinuous (l.s.c.); see, e.g., \cite[Corollary~14.34]{rw}.
The normal integrand $\ph$ is {\em proper} if the function $\ph_t$ is proper for each $t\in T$. If furthermore $\ph_t$ is a convex for all $t\in T$ that $\ph$ is called to be a {\em convex normal
integrand}.\vspace*{0.03in}

Motivated by the above definition of normal integrands for extended-real-valued functions, we introduce now its counterpart for set-valued mappings.\vspace*{-0.05in}
\begin{definition}[\bf set-valued normal integrands]\label{norm-integ} We say that a mapping $\Phi\colon T\times\X\tto\Y$ is a {\em set-valued normal integrand} on a measure space $(T,\mathcal{A},\mu)$
if for all $t\in T$ the multifunction $\Phi_t:=\Phi(t,\cdot)$ has closed graph and the graph of $\Phi$ belongs to $\mathcal{A}\otimes\mathcal{B}(\X\times\Y)$. If in addition the set $\gph\Phi_t$ is
convex for a.e.\ $t\in T$, then we say that $\Phi$ is a {\em set-valued convex normal integrand}.
\end{definition}\vspace*{-0.05in}

Based on our previous discussions and the completeness of the measure space $(T,\mathcal{A},\mu)$, we conclude that Definition~\ref{norm-integ} of set-valued normal integrands amounts to saying that
$t\mapsto\gph\Phi_t$ is a measurable multifunction with closed values.\vspace*{0.05in}

Next we present some results on measurable multifunctions and normal integrands that are broadly used in what follows. The first proposition concerns graph measurability of subgradient mappings generated
by extended-real-valued normal integrands as well as coderivatives associated with set-valued normal integrand mappings.\vspace*{-0.05in}
\begin{proposition}[\bf graph measurability of subgradient and coderivative mappings]\label{lemma_measurability_reg_sub} Let $\ph\colon T\times\X\to\Rex$ be a proper normal integrand, and $\Phi\colon
T\times\X\tto\Y$ be a proper set-valued normal integrand. Then the following multifunctions are graph measurable:\vspace*{-0.05in}
\begin{enumerate}[label=\alph*)]
\item[\bf(i)] $t\mapsto\gph\Hat\partial\ph_t=\big\{(x,x^\ast)\in\mathbb{R}^{2n}\big|\;x^\ast\in\Hat\partial\ph_t(x)\big\}$.
\item[\bf(ii)] $t\mapsto\gph\partial\ph_t:=\big\{(x,x^\ast)\in\mathbb{R}^{2n}\;\big|\;x^\ast\in\partial\ph_t(x)\big\}$.
\item[\bf(iii)] $t\mapsto\gph\Hat D^\ast\Phi_t:=\big\{(x,y,x^\ast,y^\ast)\in\mathbb{R}^{2(n+m)}\big|\;x^\ast\in\Hat D^\ast\Phi_t(x,y)(y^\ast)\big\}$.
\item[\bf(iv)] $t\mapsto\gph D^\ast\Phi_t:=\big\{(x,y,x^\ast,y^\ast)\in\mathbb{R}^{2(n+m)}\;\big|\;x^\ast\in D^\ast\Phi_t(x,y)(y^\ast)\big\}$.
\end{enumerate}
\end{proposition}\vspace*{-0.05in}
{\bf Proof}. Item (i) is proved in \cite[Theorem~3.2]{mp20}, item (ii) follows from (i), but it can be also derived from \cite[Theorems~14.26 and 14.60]{rw}. Finally, items (iii) and (iv) follow from (i)
and (ii), respectively, by their applications to the normal integrand $(t,x,y)\mapsto\delta_{\gph\Phi_t}(x,y)$. $\h$\vspace*{0.05in}

The following result is classical in the theory of measurable multifunctions. It asserts that a graph measurable multifunction with nonempty, while not necessarily closed, values admits a measurable
(single-valued) selection; see, e.g., the book \cite[Theorem~III.22]{cv} and its references.\vspace*{-0.05in}
\begin{proposition}[\bf measurable selections]\label{Prop_measurableselection} Let $F\colon T\tto\X$ be a graph measurable multifunction with nonempty values on a measurable space $(T,\mathcal{A},\mu)$.
Then there exists a measurable single-valued mapping $\x\colon T\to\X$ such that $\x(t)\in F(t)$ for a.e.\ $t\in T$.
\end{proposition}\vspace*{-0.05in}

Finally in this section, we recall the notion of expected-integral functionals introduced in our preceding paper \cite{mp20} and present two sequential Leibniz-type rules in terms of regular subgradients
that were established therein. Let $\varphi\colon T\times\X\times\Y\to\Rex$ be a proper normal integrand such that there exists $\nu\in\Leb^1(T,\mathbb{R}_+)$ ensuring the {\em uniform boundedness from
below} condition
\begin{align}\label{lower_bound_assump}
\varphi_t(v,w)\ge-\nu(t)\;\text{ for all }\;v\in\X\;\;w\in\Y\text{ and all  } t\in T.
\end{align}
As defined in \eqref{eif}, the {\rm expected-integral} functional $\Intf{\varphi}\colon\X\times\Leb^1(\T,\Y)\to\Rex$ generated by a normal integrand $\ph$ is given by the formula
\begin{align*}
\Intf{\varphi}(x,\y):=\int_{\T}\varphi_t(x,\y(t))d\mu.
\end{align*}
It is well known in measure theory (see, e.g., \cite{bog}) that in a finite measure space $(T,\mathcal{A},\mu)$ there exist measurable disjoint sets $T_{pa}$ and $T_{na}$ such that $\mu_{pa}(\cdot):
=\mu(\cdot\cap T_{pa})$ is purely atomic and $\mu_{na}(\cdot):=\mu(\cdot\cap T_{na})$ is nonatomic. Moreover, $T_{pa}$ is a countable union of disjoint atoms.

In the following results we assume that for a given point of interest $\ox$ there exists $\rho>0$ such that
\begin{equation}\label{convex_cond}
\varphi_t(v,\cdot)\;\text{ is convex for all }\;v\in \mathbb{B}_\rho(\ox)\;\text{ and }\;t\in T_{na}.
\end{equation}

Now we are ready to present two {\em sequential subdifferential Leibniz rules} for expected-integral functionals taken, respectively, from Theorem~5.2 and Theorem~5.4 of our preceding paper
\cite{mp20}. Note that, although these results are not written in the expected Leibniz form, they certainly can be treated as {\em approximate} versions. Furthermore, such results will lead us to
the desired generalized Leibniz rules by employing limiting procedures under appropriate qualification conditions. \vspace*{-0.05in} \begin{theorem}[\bf subdifferential Leibniz rule,
I]\label{theoremsubdiferential} Let $\ph$ be a proper normal integrand on $(T,\mathcal{A},\mu)$ satisfying \eqref{lower_bound_assump} and \eqref{convex_cond} around some $\ox\in\X$. Pick any
$p,q\in(1,\infty)$ with $1/p+1/q=1$ and suppose that $(\ox^*,\ooy^\ast)\in\Hat{\partial}\Intf{\varphi}(\ox,\ooy)$ are such that the function \begin{equation*} t\mapsto\inf_{\X
\times\Y}\big\{\varphi_t(\cdot,\cdot)-\la\ooy^\ast(t),\cdot\ra\big\} \end{equation*} is integrable on $T$. Then there exist sequences $\{x_k\}\subset\X$, $\{\x_k\}\subset\Leb^p({T},\X)$,
$\{\x_k^*\}\subset{\Leb}^q({T},\X)$, $\{\y_k\}\subset\Leb^1(T,\Y)$, and $\{\y_k^\ast\}\subset \Leb^\infty(T,\Y)$ satisfying the assertions: \begin{enumerate}[label=\alph*),ref=\alph*)]
\item[\bf(i)] $\big(\x_k^*(t),\y_k^\ast(t)\big)\in\Hat\partial\varphi_t\big(\x_k(t),\y_k(t)\big)$ for a.e.\ and all $k\in\N$. \item[\bf(ii)] $\|\ox-x_k\|\to 0$, $\|\ox-\x_k\|_p\to 0$, and
$\|\ooy-\y_k\|_1\to 0$ as $k\to\infty$. \item[\bf(iii)] $\disp\bigg\|\int_T \x_k^*(t)d\mu-\ox^\ast\bigg\|\to 0$ and $\|\x_k^*\|_q\|\x_k-x_k\|_p\to 0$ as $k\to\infty$. \item[\bf(iv)]
$\disp\int_T\Big|\varphi_t\big(\x_k(t),\y_k(t)\big)-\varphi_t\big(\ox,\ooy(t)\big)\Big|d\mu\to 0$ as $k\to\infty$. \item[\bf(v)] $\|\y_k^\ast-\ooy^\ast\|_\infty\to 0$ as $k\to\infty$.
\end{enumerate} \end{theorem}\vspace*{-0.07in} \begin{theorem}[\bf subdifferential Leibniz rule, II]\label{theorem_main_fuzzy_sub}\hspace{-0.2cm} Let $\ph$ be a proper normal integrand on
$(T,\mathcal{A},\mu)$ satisfying \eqref{lower_bound_assump} and \eqref{convex_cond} at $\ox\in\X$, and let $(\ox^*,\ooy^\ast)\in\Hat{\partial}\Intf{\varphi}(\ox,\ooy)$ be chosen so that the
function \begin{equation*}\label{inf} t\mapsto\inf_{\mathbb{B}_{\hat\rho}(\ox)\times\Y}\big\{\varphi_t(\cdot,\cdot)-\la\ooy^\ast(t),\cdot\ra\big\} \end{equation*} is integrable on $T$ for some
$\hat{\rho}>0$. Then there exist sequences $\{x_k\}\subset\X$, $\{\x_k\}\subset\Leb^\infty({T},\X)$, $\{\x_k^\ast\}\subset{\Leb}^1({T},\X)$, $\{\y_k\}\subset\Leb^1(T,\Y)$, and
$\{\y_k^\ast\}\subset\Leb^\infty(T,\Y)$ such that: \begin{enumerate}[label=\alph*),ref=\alph*)] \item[\bf(i)] $\big(\x_k^*(t),\y_k^\ast(t)\big)\in\Hat\partial\varphi_t\big(\x_k(t),\y_k(t)\big)$
for a.e.\ and all $k\in\N$. \item[\bf(ii)] $\|\ox-x_k\|\to 0$, $\|\ox-\x_k\|_\infty\to 0$, and $\|\ooy-\y_k\|_1\to 0$ as $k\to\infty$. \item[\bf(iii)]
$\disp\bigg\|\int_T\x_k^*(t)d\mu-\ox^\ast\bigg\|\to 0$, and $\disp\int_T\Big\|\x_k^*(t)\Big\|\cdot\Big\|\x_k(t)-x_k\Big\|d\mu\to 0$ as $k\to\infty$. \item[\bf(iv)]
$\disp\int_T\Big|\varphi_t\big(\x_k(t),\y_k(t)\big)-\varphi_t(\ox,\ooy(t)\big)\Big|d\mu\to 0$ as $k\to\infty$. \item[\bf(v)] $\|\y_k^\ast-\ooy^\ast \|_\infty\to 0$ as $k\to\infty$.
\end{enumerate} \end{theorem}\vspace*{-0.25in}

\section{Lipschitzian Properties of Random Multifunctions}\label{sec3}\sce\vspace*{-0.1in}

This section is devoted to the study of new Lipschitzian properties for random multifunctions defined on finite measure spaces. Our main attention is paid to the three major properties of such
multifunctions that we label as {\em integrably local Lipschitzian}, {\em integrably quasi-Lipschitzian}, and {\em integrably Lipschitz-like} ones. We reveal relationships between these
properties in various classes of measure spaces and compare them with the corresponding properties of deterministic multifunctions. On the one hand, the new Lipschitzian properties are important
to establish {\em stability} of random sets of feasible solutions in stochastic programming, but on the other hand they play a crucial role in deriving {\em calculus rules} to evaluate
coderivatives of expected-integral multifunctions that is given in the subsequent sections.\vspace*{0.03in}

Starting with the local Lipschitzian property of deterministic multifunctions, recall that $\Phi\colon\X\tto\Y$ is (Hausdorff) {\em locally Lipschitzian} around $\ox\in\dom\Phi$ if there are numbers
$\eta>0$ and $\ell\ge 0$ such that
\begin{equation}\label{loc-lip}
\Phi(x)\subset\Phi(x')+\ell\|x-x'\|\mathbb{B}\;\text{ for all }\;x,x'\in\mathbb{B}_\eta(\ox).
\end{equation}

Having \eqref{loc-lip} in mind, we now define its random version for set-valued normal integrands on the general measure spaces under consideration.\vspace*{-0.05in} \begin{definition}[\bf
integrable local Lipschitzian property of random multifunctions]\label{integ-lip} Let $\Phi\colon T\times\X\tto\Y$ be a set-valued normal integrand on a complete finite measure space
$(T,\mathcal{A},\mu)$, and let $\ox\in\dom\Intfset{\Phi}$ be given  for $E_\Phi$ taken from \eqref{def:set-valued:Exp}. We say that $\Phi$ is {\em integrably locally Lipschitzian} around $\ox$ if
there exist $\eta>0$, $\ell\in\Leb^1(T,\R_+)$, and $\Hat T\in\mathcal{A}$ with $\mu(T\backslash\Hat T)=0$ such that \begin{equation}\label{eq_definition_Int_loc}
\Phi_t(x)\subset\Phi_t(x')+\ell(t)\|x-x'\|\mathbb{B}\;\text{ for  all }\;t\in\Hat T\;\text{ and }\;x,x'\in\mathbb{B}_\eta(\ox). \end{equation} \end{definition}

Let us present useful characterizations of integrably locally Lipschitzian multifunctions in terms of distance functions. Recall that the {\em Pompeiu-Hausdorff distance} between sets $\Omega_1,
\Omega_2\subset\X$ is given by \begin{equation}\label{def:PH:Dist} \begin{aligned} {\rm haus}(\Omega_1,\Omega_2)&:=\max\left\{\sup\limits_{x\in\Omega_1}{\rm
dist}(x;\Omega_2),\sup\limits_{x\in\Omega_2}{\rm dist}(x;\Omega_1)\right\}\\ &=\sup\limits_{x\in\X}\big|{\rm dist}(x;\Omega_1)-{\rm dist}(x;\Omega_2)\big|, \end{aligned} \end{equation} where
${\rm dist}(x;\Omega)$ is the standard distance function in $\X$ with the convention that $\dist(x,\emptyset):=\infty$. Here are the equivalent descriptions of the integrable local Lipschitzian
property from \eqref{eq_definition_Int_loc}.\vspace*{-0.05in} \begin{proposition}[\bf distance descriptions of integrable local Lipschitzian multifunctions]\label{prop:dist:car} Let $\Phi\colon
T\times\X\tto\Y$ be a set-valued normal integrand, and let $\ox\in\dom\Intfset{\Phi}$ for $\Intfset{\Phi}$ taken from \eqref{def:set-valued:Exp}. Then the following assertions are equivalent:
\begin{enumerate}[label=\alph*),ref=\alph*)]\vspace*{-0.05in} \item[\bf(i)] There exist $\Hat T\in\mathcal{A}$ with $\mu(T\backslash\Hat T)=0$, $\eta>0$, and $\ell\in\Leb^1(T,\R_+)$ such that for
all $t\in\Hat T$ and all $x\in\mathbb{B}_\eta$ we have the inclusion in \eqref{eq_definition_Int_loc}. \item[\bf(ii)] There exist $\Hat T\in\mathcal{A}$ with $\mu(T\backslash\Hat T)=0$, $\eta>0$,
and $\ell\in\Leb^1(T,\R_+)$ such that \begin{align}\label{Hausd:Car} {\rm haus}\big(\Phi_t(x),\Phi_t(x')\big)\le\ell(t)\|x-x'\|\;\text{ for all }\;x,x'\in\mathbb{B}_\eta(\ox),\text{ }t\in\Hat T.
\end{align} \item[\bf(iii)] In the setting of {\rm(i)} we have that the function $x\mapsto{\rm dist}(y;\Phi_t(x))$ is $\ell(t)$-Lipschitz continuous on $\mathbb{B}_\eta(\ox)$. \item[\bf(iv)] In
the setting {\rm(ii)} we have the estimate \begin{equation}\label{equivdistance} {\rm dist}\big(y;\Phi_t(x)\big)\le\ell(t){\rm dist}\big(x;\Phi_t^{-1}(y)\cap\mathbb{B}_\eta(\ox)\big)\;\text{ for
all }\;x\in\mathbb{B}_\eta(\ox),\quad y\in\Y,\;\text{ and  }\;t\in\Hat T. \end{equation} \end{enumerate} \end{proposition} {\bf Proof}. Note that assertion (i) means that $\Phi$ is integrably
locally Lipschitzian around $\ox$ by Definition~\ref{integ-lip}. The equivalence between (i) and (ii) easily follows from \eqref{eq_definition_Int_loc} and definition \eqref{def:PH:Dist} of the
Pompeiu-Hausdorff distance. The representation of the latter distance presented in \eqref{def:PH:Dist} ensures the equivalence between (ii) and (iii). To finish the proof, it remains verifying
the equivalence between (ii) and (iv). Observe first that we can always suppose that $\mathbb{B}_\eta(\ox)\subset\dom\Phi_t$ for all $t\in\Hat T$. Indeed, the choice of $\ox$ tells us that
$\ox\in \dom\Phi_t$ holds for a.e.\ $t\in T$. Hence it follows from (ii) that ${\rm haus}(\,\Phi_t(x),\Phi_t(\ox))$ is finite for all $x\in\mathbb{B}_\eta(\ox)$ and $t\in\Hat{T}$, which means
that $\mathbb{B}_\eta(\ox)\subset\dom\Phi_t$ whenever $t\in\Hat T$. On the other hand, assuming (iv) gives us $t\in \Hat{T}$, $y_t\in\Phi_t(\ox)$, and $x\in\mathbb{B}_\eta(\ox)$ such that
\eqref{equivdistance} yields ${\rm dist}\big(y_t;\Phi_t(x)\big)\le\ell(t)\|x-\ox\|$, which shows the necessity of $\mathbb{B}_\eta(\ox)\subset\dom\Phi_t$. The equivalence between (ii) and (iv)
follows now from \cite[Proposition~3C.1]{dr}, which therefore completes the proof of this proposition.\vspace*{0.05in}

The following example illustrates how the the distance characterizations of Proposition~\ref{prop:dist:car} allow us to easily check the fulfillment of the integrable local Lipschitzian property
of random multifunctions.\vspace*{-0.05in} \begin{example} $(${\em checking the integrable locally Lipschitzian property of multifunctions}$)$\label{exa} Let $F\colon T\tto\Y$, be an integrable
bounded measurable multifunction, i.e., there exists $\lambda\in\Leb^1(T,\R_+)$ such that $F(t)\subset\lambda(t)\mathbb{B}$, and let $b\colon T\times\X\to\Y$ and $A\colon
T\times\X\to\mathbb{R}^{m\times m}$ be two measurable mappings. Take $\bar x\in\X$ for which there are $\eta>0$, $\ell_1\in\R_+$, and $\ell_2\in\Leb^1(T,\R_+)$ ensuring that \begin{equation*}
\|A(t,x)-A(t,x')\|\le\ell_1\|x-x'\|\;\mbox{ and }\;\|b(t,x)-b(t,x')\|\le\ell_2(t)\|x-x'\| \end{equation*} for all $x,x'\in\mathbb{B}_\eta(\ox),\;t\in T$, and that $A(t,x)$ is nonsingular for such
$t,x$. Then the mapping $\Phi_t(x):=A(t,x)F(t)+b(t,x)$ is integrably locally Lipschitzian around $\ox$. Indeed, it follows from \cite[Example~9.32]{rw} that for a.e.\ $t\in T$ the modulus of
Lipschitz continuity of $\Phi_t$ on $\mathbb{B}_\eta(\ox)$ can be bounded by $\lambda(t)\ell_1+\ell_2(t)$. Thus we conclude from the distance characterizations \eqref{Hausd:Car} of
Proposition~\ref{prop:dist:car} that the above multifunction $\Phi$ enjoys the claimed Lipschitzian property. \end{example}\vspace*{-0.05in}

Our next goal is to establish {\em coderivative characterizations} of integrably locally Lipschitzian and related Lipschitzian properties of random multifunctions. Recall first that for {\em
deterministic} set-valued mappings $\Phi\colon\X\tto\Y$ the local Lipschitzian property \eqref{loc-lip} can be written in form \eqref{eq_definition_Int_loc}, where the measure space is a {\em
single atom}. The latter property is characterized by \begin{equation}\label{cod-loclip} \sup\big\{\|x^*\|\;\big|\;x^*\in D^*\Phi(\ox,\oy)(y^*)\big\}\le\ell\|y^*\|\;\mbox{ for all
}\;\oy\in\Phi(\ox),\;y^*\in\Y, \end{equation} provided that $\Phi$ is closed-graph and is {\em uniformly bounded} around this point; see \cite[Theorem~5.11]{m93}. Due the robustness of the
coderivative \eqref{lcod} and its representation \eqref{Lim_repr_cod}, the pointwise characterization in \eqref{cod-loclip} can be equivalently written in the following forms: there exist
$\eta>0$ and $\ell\ge 0$ such that \begin{equation}\label{cod-loclip1} \sup\big\{\|x^*\|\;\big|\;x^*\in D^*\Phi(x,y)(y^*)\big\}\le\ell\|y^*\|\;\mbox{ and }\;\sup\big\{\|x^*\|\;\big|\;x^*\in\Hat
D^*\Phi(x,y)(y^*)\big\}\le\ell\|y^*\| \end{equation} for all $x\in\B_\eta(\ox)$, $y\in\Phi(x)$, and $y^*\in\Y$.

Observe further that a graphical localization of the locally Lipschitzian property \eqref{loc-lip} of deterministic multifunctions $\Phi\colon\X\tto\Y$ is known as the {\em Lipschitz-like}
(pseudo-Lipschitz, Aubin) property of $\Phi\colon\X\tto\Y$ around $(\ox,\oy)\in\gph\Phi$ defined as: there exist $\eta>0$ and $\ell\ge 0$ such that
\begin{equation}\label{aub}
\Phi(x)\cap\B_\eta(\oy)\subset\Phi(x')+\ell\|x-x'\|\mathbb{B}\;\text{ for all }\;x,x'\in\mathbb{B}_\eta(\ox).
\end{equation}
As well recognized (see, e.g., \cite[Theorem~1.42]{m06}), the locally Lipschitzian property of $\Phi$ around $\ox\in\dom\Phi$ is equivalent to the Lipschitz-like property of $\Phi$ around
$(\ox,\oy)$ for all $\oy\in\Phi(\ox)$, provided that $\Phi$ is closed-graph and uniformly bounded around $\ox$. Due to this fact, the above characterization \eqref{cod-loclip} of the locally
Lipschitzian property of deterministic multifunctions is a consequence of the following characterizations of the Lipschitz-like property of $\Phi$ around $(\ox,\oy)$ given by
\begin{equation}\label{cod-cr}
\sup\big\{\|x^*\|\;\big|\;x^*\in D^*\Phi(\ox,\oy)(y^*)\big\}\le\ell\|y^*\|\;\mbox{ for all }\;y^*\in\Y,
\end{equation}
which is known as the {\em coderivative/Mordukhovich criterion}; see \cite[Theorem~5.7]{m93} and \cite[Theorem~9.40]{rw}. Similarly to the case of locally Lipschitzian multifunctions, we can
equivalently reformulate \eqref{cod-cr} via the neighborhood estimates in \eqref{cod-loclip1} but valid now for all $x\in\B_\eta(\ox)$, $y\in\Phi(x)\cap\B_\eta(\oy)$, and $y^*\in\Y$ without the
uniform boundedness assumption on $\Phi$ around $\ox$.\vspace*{0.05in}

The situation with {\em random} multifunctions is essentially more involved in comparison with the deterministic case, being rather similar in some aspects while significantly different in the
others; this can be precisely seen from the results given below. Let us start with a coderivative characterization of integrably locally Lipschitzian multifunctions \eqref{eq_definition_Int_loc}
defined on general measure spaces. The following theorem not only extends the coderivative conditions in \eqref{cod-loclip1} to random mappings, but also provides a new result in the
deterministic case by {\em dropping the uniform boundedness} assumption. For brevity, we prove a regular coderivative characterization similar to the second estimate in \eqref{cod-loclip1}. The
one in the form of the first estimate therein can be derived similarly to the proof of Proposition~\ref{eq:regular:basic} given below.\vspace*{-0.05in} \begin{theorem}[\bf coderivative
characte	rization of integrably locally Lipschitzian multifunctions]\label{Prop47} Let $\Phi\colon T\times\X\tto\Y$ be a set-valued normal integrand defined on a complete finite measure space
$(T,\mathcal{A},\mu)$, and let $\ox\in\dom\Intfset{\Phi}$. Then $\Phi$ is integrably locally Lipschitzian around $\ox$ if and only if there exist $\eta>0$, $\ell\in\Leb^1(T,\R_+)$, and $\Hat
T\in\mathcal{A}$ with $\mu(T\backslash\Hat T)=0$ such that \begin{equation}\label{Int_Lips_like_inq_nonatomic} \sup\big\{\|x^*\|\;\big|\;x^*\in\Hat
D^*\Phi_t(x,y)(y^*)\big\}\le\ell(t)\|y^*\|\;\mbox{ for all }\;y\in\Phi_t(x)\;\mbox{ and }\;t\in\Hat T \end{equation} whenever $x\in\B_\eta(\ox)$ and $y^*\in\R^m$. \end{theorem}\vspace*{-0.05in}
{\bf Proof}. First we show that the fulfillment of condition \eqref{Int_Lips_like_inq_nonatomic} for all $x\in\B_\eta(\ox)$ and $y^*\in\R^m$ yields the integrable local Lipschitz property of
$\Phi$ around $\ox$. Take $\eta>0$, $\ell\in\Leb^1(T,\R_+)$, and $\Hat T\in\mathcal{A}$ with $\mu(T\backslash\Hat T)$ for which \eqref{Int_Lips_like_inq_nonatomic} holds, and then fix $t\in\Hat
T$. Picking $x_0,x_1\in\mathbb{B}_{\eta/2}(\ox)$ and $y\in\Phi_t(x_0)$, we claim that \begin{equation}\label{equation000} y\in\Phi_t(x_1)+\ell(t)\|x_0-x_1\|\mathbb{B}. \end{equation} To proceed,
define $x_\lambda:=(1-\lambda)x_0+\lambda x_1$ with $\lambda\in[0,1]$ and denote \begin{equation}\label{lm}
\bar\lambda:=\sup\big\{\lambda>0\;\big|\;y\in\Phi_t(x_\lambda)+\ell(t)\lambda\|x_0-x_1\|\mathbb{B}\big\}. \end{equation} Since $t$ is fixed, we can consider $\Phi_t(x)$ as a deterministic
multifunction and thus deduce from the aforementioned version of \eqref{cod-loclip1} as a characterization of the Lipschitz-like property of deterministic multifunctions by
\cite[Theorem~4.7]{m06} that $\Phi_t$ is Lipschitz-like around $(x_0,y)$. This gives us $\eta_{x_0,y}>0$ such that \begin{equation*}
y\in\Phi_t(x_0)\cap\mathbb{B}_{\eta_{x_0,y}}(y)\subset\Phi_t(u)+\ell(t)\|x_0-u\|\mathbb{B}\;\text{ for all }\;u\in\mathbb{B}_{\eta_{x_0,y}}(x_0). \end{equation*} Using the above inclusion for
$u=x_\lambda$ with a sufficiently small number $\lambda>0$ shows that $\bar\lambda>0$. Let us show that the supremum in \eqref{lm} is attained, i.e., we have \begin{equation}\label{equation001}
\exists y_{\bar{\lambda}}\in\Phi_t(x_{\bar\lambda}),\;\exists v_{\bar\lambda}\in\mathbb{B}\;\text{ such that }\;y=y_{\bar\lambda}+\ell(t)\bar\lambda\|x_0-x_1\|v_{\bar\lambda}. \end{equation}
Indeed, consider a sequence $\lambda_k\uparrow\bar\lambda$ together with the points $y_{\lambda_k}\in\Phi_t(x_{\lambda_k})$ and $v_k\in\mathbb{B}$ satisfying \begin{equation}\label{equation002}
y=y_{\lambda_k}+\ell(t){\lambda_k}\|x_0-x_1\|v_k,\quad k\in\N. \end{equation} It follows that $x_{\lambda_k}\to x_{\bar\lambda}$ as $k\to\infty$, and that $v_k\to v_{\bar \lambda}\in\mathbb{B}$
along a subsequence. Thus we deduce from \eqref{equation002} that a subsequence of $\{y_{\lambda_k}\}$ converges to some $y_{\bar\lambda}$, which yields \eqref{equation001} by the closedness of
$\gph\Phi_t$.

Applying now \eqref{Int_Lips_like_inq_nonatomic} and the characterization of \cite[Theorem~4.7]{m06} tells us that $\Phi_t$ is Lipschitz-like around $(x_{\bar \lambda},y_{\bar \lambda})$, i.e.,
there exists $\eta_{\bar\lambda}>0$ ensuring \begin{equation*}
y_{\bar\lambda}\in\Phi_t(x_{\bar\lambda})\cap\mathbb{B}_{\eta_{\bar\lambda}}(y_{\bar\lambda})\subset\Phi_t(w)+\ell(t)\|w-x_{\bar\lambda}\|\mathbb{B}\;\mbox{ for all
}\;w\in\mathbb{B}_{\eta_{\bar\lambda}} (x_{\bar\lambda}). \end{equation*} The latter allows us to find $\lambda'>\bar\lambda$ for which $x_{\lambda'}\in\mathbb{B}(x_{\bar\lambda})$,
$y_{\lambda'}\in\Phi_t(x_{\lambda'})$, and $v_{\lambda'} \in\mathbb{B}$ are such that \begin{equation}\label{equation004} y_{\bar\lambda}\in
y_{\lambda'}+\ell(t)\|x_{\lambda'}-x_{\bar\lambda}\|v_{\lambda'}=y_{\lambda'}+\ell(t)(\lambda'-\bar\lambda)\|x_1-x_0\|v_{\lambda'}. \end{equation} It readily follows from \eqref{equation001} and
\eqref{equation004} that \begin{align*} y&=y_{\bar\lambda}+\ell(t)\bar\lambda\|x_0-x_1\|v_{\bar\lambda}\\
&=y_{\lambda'}+\ell(t)(\lambda'-\bar\lambda)\|x_1-x_0\|v_{\lambda'}+\ell(t)\bar\lambda\|x_0-x_1\|v_{\bar\lambda}\\
&=y_{\lambda'}+\ell(t)\lambda'\|x_1-x_0\|\left(\frac{\lambda'-\bar\lambda}{\lambda'}v_{\lambda'}+\frac{\bar\lambda}{\lambda'}v_{\bar\lambda}\right)\\
&\in\Phi_t(x_{\lambda'})+\ell(t)\lambda'\|x_0-x_1\|\mathbb{B}, \end{align*} which contradicts the maximality of $\bar\lambda$ as chosen in \eqref{lm}. This gives us \eqref{equation000} and
verifies therefore that $\Phi$ is integrably locally Lipschitzian around $\ox$.

Conversely, suppose that $\Phi$ is integrably locally Lipschitzian around $\ox$ and thus find $\eta>0$, $\ell\in\Leb^1(T,\R_+)$, and $\Hat T\in\mathcal{A}$ with $\mu(T \backslash\Hat T)$ such
that \eqref{eq_definition_Int_loc} is satisfied. Taking $x\in\mathbb{B}_{\eta/2}(\ox)$, $t\in\Hat T$, and $y\in\Phi_t(x)$, we deduce from \eqref{eq_definition_Int_loc} and the characterization of
\cite[Theorem~4.7]{m06} that \begin{equation*} \sup\big\{\|x^\ast\|\;\big|\;x^\ast\in\Hat{D}^\ast\Phi_t(x,y)(y^\ast)\big\}\le\ell(t)\|y^\ast\|\;\mbox{ for all }\;y^\ast\in\Y, \end{equation*}
which concludes the proof of the theorem. $\h$\vspace*{0.05in}

Now we are ready to introduce two {\em graphically localized} Lipschitzian properties of multifunctions defined on finite measure spaces. They both may be considered as extensions to random
multifunctions of the Lipschitz-like property \eqref{aub} and its coderivative characterizations for deterministic multifunctions while being generally different from each other in the stochastic
case.

To proceed, we associate with a given set-valued integrand $\Phi\colon T\times\X\tto\Y$ the multifunction ${\cal S}\colon\X\times\Y\to\Leb^1(T,\Y)$ defined by
\begin{equation}\label{mapping:SPHI}
\mathcal{S}_{\Phi}(x,y):=\Big\{\y\in\Leb^1(T,\Y)\;\Big|\;\int_T\y(t)d\mu=y\;\text{ and }\;\y(t)\in\Phi_t(x)\;\text{ for a.e. }\;t\in T\Big\}.
\end{equation}\vspace*{-0.2in}
\begin{definition}[\bf graphically localized random Lipschitzian multifunctions]\label{int-lipl} Let $\Phi\colon T\times\X\tto\Y$ be a set-valued normal integrand on a complete finite measure space
$(T,\mathcal{A},\mu)$, let $(\ox,\oy)\in\gph\Intfset{\Phi}$ for $E_\Phi$ taken from \eqref{def:set-valued:Exp}, and let $\ooy\in\mathcal{S}_{\Phi}(\ox,\oy)$ for ${\cal S}_\Phi$ taken from
\eqref{mapping:SPHI}. We say that:

{\bf (i)} $\Phi$ is {\em integrably Lipschitz-like} around $(\ox,\ooy)$ if there exist $\ell,\in\Leb^1(T,\R_+)$, strictly positive measurable functions  $\eta$ and $\gg$, and a measurable set
$\Hat T\in\mathcal{A}$ with $\mu(T\backslash\Hat T)=0$ such that \begin{equation}\label{eq_definition_Int_loc:Lipschitz-like}
\Phi_t(x)\cap\B_{\gg(t)}\big(\y(t)\big)\subset\Phi_t(x')+\ell(t)\|x-x'\|\mathbb{B}\;\text{ for  all }\;t\in\Hat T\;\text{ and }\;x,x'\in\mathbb{B}_{\eta(t)}(\ox). \end{equation} {\bf(ii)} $\Phi$
is {\em integrably quasi-Lipschitzian} around $(\ox,\ooy)$ if there exist $\eta>0$ and $\ell\in\Leb^1(T,\R_+)$ such that \begin{equation}\label{Int_Lips_like_inq}
\sup\big\{\|x^\ast\|\;\big|\;x^\ast\in{D}^\ast\Phi_t\big(\x(t),\y(t)\big)\big(\y^\ast(t)\big)\big\}\le\ell(t)\|\y^\ast(t)\|\text{ for a.e. }\;t\in T \end{equation} whenever
$\x\in\mathbb{B}_\eta(\ox)$, $\y\in\mathbb{B}_\eta(\ooy)\cap\Phi(\x)$, and $\y^\ast\in\Leb^\infty(T,\Y)$ with \begin{equation*}
\mathbb{B}_\eta(\ooy)\cap\Phi(\x):=\big\{\y\in\Leb^1(T,\Y)\;\big|\;\y\in\mathbb{B}_\eta(\ooy)\;\text{ and }\;\y(t)\in\Phi_t\big(\x(t)\big)\;\text{ a.e.}\big\}. \end{equation*} \end{definition}

First we show that the integrable quasi-Lipschitzian property can be equivalently reformulated in terms of the regular coderivative \eqref{rcod} replacing the basic one in
\eqref{Int_Lips_like_inq}. \begin{proposition}[\bf equivalent description of the integrable quasi-Lipschitzian property]\label{eq:regular:basic} Let $\Phi\colon T\times\X\tto\Y$ be a set-valued
normal integrand with $(\ox,\oy)\in\gph\Intfset{\Phi}$ and $\ooy\in\mathcal{S}_{\Phi}(\ox,\oy)$. Then $\Phi$ is integrably quasi-Lipschitzian around $(\ox,\ooy)$ if and only if there exist
$\eta>0$ and $\ell\in\Leb^1(T,\R_+)$ such that \begin{equation}\label{Regular:Int_Lips_like_inq}
\sup\big\{\|x^\ast\|\;\big|\;x^\ast\in\Hat{D}^\ast\Phi_t\big(\x(t),\y(t)\big)\big(\y^\ast(t)\big)\big\}\le\ell(t)\|\y^\ast(t)\|\;\text{ for a.e. }\;t\in T \end{equation} whenever
$\x\in\mathbb{B}_\eta(\ox)$, $\y\in\mathbb{B}_\eta(\ooy)\cap\Phi(\x)$, and $\y^\ast\in\Leb^\infty(T,\Y)$. \end{proposition} {\bf Proof}. We obviously have that \eqref{eq_definition_Int_loc}
yields \eqref{Regular:Int_Lips_like_inq}. To verify the opposite implication, suppose without loss of generality that $\mu(T)=1$ and take $\eta>0$ and $\ell\in\Leb^1(T,\R_+)$ from
\eqref{Regular:Int_Lips_like_inq}. Fix further $\gamma\in(0,\eta)$ and pick $\x\in\mathbb{B}_{\gamma}(\ox)$, $\y\in\mathbb{B}_\gamma(\ooy)\cap\Phi(\x)$, and $\y^\ast\in\Leb^\infty(T,\Y)$.
Proposition~\ref{lemma_measurability_reg_sub} tells us that the set-valued mapping $F(t):={D}^\ast\Phi_t(\x(t),\y(t))(\y^\ast(t))$ is measurable with closed values on $T$. By the {\em Castaing
representation} of the measurable multifunction $F$ (see, e.g., \cite[Theorem~14.5]{rw}) we find a sequence of measurable selections $\x_k^\ast(t)\in F(t)$ for a.e.\ $t\in\Hat T:=\dom F$ such
that \begin{equation*} \sup\big\{\|x_k^\ast(t)\|\;\big|\;k\in\N\big\}=\sup\big\{\|x^\ast\|\;\big|\;x^\ast\in{D}^\ast\Phi_t\big(\x(t),\y(t)\big)\big(\y^\ast(t)\big)\big\}\;\text{ for a.e.
}\;t\in\Hat T. \end{equation*} Fixing any $\epsilon\in(0,(\eta-\gamma)/2)$ and $k\in\N$, define the multifunction $F^k_\epsilon\colon\Hat T\tto\R^{2(n+m)}$ by \begin{equation*}
(u,u^\ast,v,v^\ast)\in F^k_\epsilon(t)\Longleftrightarrow \left\{\begin{array}{ll}
u^\ast\in\Hat{D}^\ast\Phi_t(u,v)(v^\ast),\;u\in\mathbb{B}_\gamma\big(\x(t)\big),\;v\in\mathbb{B}_{\gamma}\big(\y(t)\big),\\
\|u^\ast-\x_k^\ast(t)\|\le\epsilon,\;\|v^\ast-\ooy^\ast(t)\|\le\epsilon/(1+\ell(t)). \end{array}\right. \end{equation*} It follows from Proposition~\ref{lemma_measurability_reg_sub} and the
measurability of the mappings involved that the multifunctions $F^k_\epsilon$ are graph measurable. Applying  the coderivative representation \eqref{Lim_repr_cod} to $F_\epsilon^k(\cdot)$ ensures
that the sets $F_\epsilon^k(t)$ are nonempty for a.e.\ $t\in\Hat T$. Then Proposition~\ref{Prop_measurableselection} gives us a measurable selection $(\u(t),\u^\ast(t),\v(t),\v^\ast(t))\in
F_\epsilon^k(t)$ for a.e.\ $t\in\Hat T$ and $\ve,k$ fixed above. Thus $\u\in\mathbb{B}_\eta(\ox)$, $\v\in\mathbb{B}_\eta(\ooy)\cap\Phi(\u)$, and $\v^\ast\in\Leb^\infty(T,\Y)$. Employing the
regular coderivative estimate \eqref{Regular:Int_Lips_like_inq} ensures the relationships \begin{align*}
\|\x_k^\ast(t)\|&\le\|\u^\ast(t)\|+\epsilon\le\ell(t)\|\v^\ast(t)\|+\epsilon\le\ell(t)\|\ooy^\ast(t)\|+\ell(t)\|\v^\ast(t)-\ooy^\ast(t)\|+\epsilon\\ &\le\ell(t)\|\ooy^\ast(t)\|+2\epsilon\;\mbox{
a.e. }\;t\in\Hat T. \end{align*} Since $\epsilon\in(0,(\eta-\gamma)/2)$ and $k\in\N$ were chosen arbitrarily, we arrive at the integrable quasi-Lipschitzian property \eqref{Int_Lips_like_inq} and
thus complete the proof of the proposition. $\h$\vspace*{0.08in}

Next we establish closed relationships (actually the equivalence) between the graphically localized Lipschitzian properties of random multifunctions from Definition~\ref{int-lipl} in the case of
{\em purely atomic} spaces with countably many atoms. The following theorem can be treated as a stochastic extension of the coderivative characterization of the Lipschitz-like property for
deterministic multifunctions.\vspace*{-0.05in}\begin{theorem}[\bf relationships between graphically localized integrably Lipschitzian properties]\label{Prop4700} Let $(T,\mu,\mathcal{A})$ be a
purely atomic measure space consisting of countable disjoint family of atoms $(T_k)_{k\in\N}$. Consider a set-valued normal integrand $\Phi\colon T\times\X\tto\Y$ with
$(\ox,\oy)\in\gph\Intfset{\Phi}$ and $\ooy\in\mathcal{S}_{\Phi}(\ox,\oy)$. If $\Phi$ is integrably quasi-Lipschitzian around $(\ox,\ooy)$, then there exist $\gg>0$ and $\ell\in\Leb^1(T,\R_+)$
such that for all $k\in\N$ we have the integrable Lipschitz-like property in the form \begin{equation}\label{eq_definition_Int_locPurely} \Phi_t(x)\cap\mathbb{B}_{\frac{\gamma}{\mu(T_k)
}}\big(\oy(t)\big)\subset\Phi_t(x')+\ell(t)\|x-x'\|\mathbb{B}\; \text{ for  all }\;x,x'\in\mathbb{B}_{\frac{\gamma}{(\ell(t)+1)}} (\ox) \;\text{ and } \;t\in T_k. \end{equation} Conversely, the
fulfillment of the integrable Lipschitz-like property \eqref{eq_definition_Int_locPurely} with some $\ell\in\Leb^\infty(T,\R_+)$ implies that $\Phi$ is integrably quasi-Lipschitzian around
$(\ox,\oy)$. \end{theorem}\vspace*{-0.05in} {\bf Proof}. Suppose that $\Phi$ is integrably quasi-Lipschitzian around $(\ox,\ooy)$ with the given data $\eta>0$ and $\ell\in\Leb^1(T,\R_+)$ for
which \eqref{Int_Lips_like_inq} holds. Since the measure is purely atomic, each measurable function or multifunction must be constant a.e.\ on each atom. Then we can assume that all the involved
mappings are constants over the atoms, i.e., $\Phi_t=\Phi_k$, $\ooy(t)=\oy_k$, and $\ell(t)=\ell_k$ on $T_k$. Furthermore, observe that the coderivative condition
\eqref{Regular:Int_Lips_like_inq} implies that for every $k\in\N$ we have \begin{align}\label{equation02}
\sup\big\{\|x^\ast\|\;\big|\;x^\ast\in\Hat{D}^\ast\Phi_k(x,y)(y^\ast)\big\}\le\ell_k\|y^\ast\| \end{align} whenever $x\in\mathbb{B}_\eta(\ox_k)$,
$y\in\mathbb{B}_{\frac{\eta}{\mu(T_k)}}(\oy_k)\cap\Phi_k(x)$, and $y^\ast\in\Y$. Denote $\gamma:=\frac{\eta}{3(\mu(T)+1)}$ and fix $k\in\N$. Then consider
$x_0,x_1\in\mathbb{B}_{\frac{\gamma}{(\ell_k+1)}}(\ox)$ and $y_k\in\Phi_k(x_0)\cap\mathbb{B}_{\frac{\gamma}{\mu(T_k)}}(\oy_k)$. We claim that \begin{equation}\label{equation000Purely}
y_k\in\Phi_k(x_1)+\ell_k\|x_0-x_1\|\mathbb{B}. \end{equation} Indeed, define $x_\lambda:=(1-\lambda)x_0+\lambda x_1$ for $\lambda\in[0,1]$ and \begin{equation*}
\bar{\lambda}:=\sup\big\{\lambda>0\;\big|\;y_k\in\Phi_k(x_\lambda)+\ell_k\lambda\|x_0-x_1\|\mathbb{B}\big\}. \end{equation*} Applying the coderivative criterion \eqref{equation02} for the
deterministic multifunctions $\Phi_k$, $k\in\N$, we have by \cite[Theorem~4.7]{m06} that $\Phi_k$ is Lipschitz-like around $(x_0,y_k)$. This gives us $\eta_{k}>0$ such that
\begin{equation}\label{equation002Purely} y_k\in\Phi_k(x_0)\cap\mathbb{B}_{\eta_{k}}(y_k)\subset\Phi_k(u)+\ell_k\|x_0-u\|\mathbb{B}\;\text{ for all }\;u\in\mathbb{B}_{\eta_{k}}(x_0).
\end{equation} Plugging there $u=x_\lambda$ for sufficiently small $\lambda>0$ ensures that $\bar{\lambda}>0$. Arguing similarly to the proof of Theorem~\ref{Prop47} verifies that the supremum in
the definition of $\bar{\lambda}$ is realized, i.e., \begin{equation}\label{equation001Purely} \exists y_{\bar{\lambda}}\in\Phi_k(x_{\bar{\lambda}}),\;\exists
v_{\bar{\lambda}}\in\mathbb{B}\;\text{ such that }\;y_k=y_{\bar\lambda}+\ell_k\bar\lambda_k\|x_0-x_1\|v_{\bar\lambda}. \end{equation} Moreover, we have the following estimate of the distance
between $\oy_k$ and $y_{\bar{\lambda}}$ from \eqref{equation001Purely}: \begin{equation*} \|\oy_k-y_{\bar{\lambda}}\|\leq \frac{2}{3\mu(T_k)}\eta. \end{equation*} Using \eqref{equation02} and
applying again \cite[Theorem~4.7]{m06} tell us that each multifunction $\Phi_k$ is Lipschitz-like around $(x_{\bar\lambda},y_{\bar\lambda})$. Thus there exists $\eta_{\bar\lambda}>0$ ensuring the
inclusions \begin{equation}\label{equation003purely}
y_{\bar\lambda}\in\Phi_k(x_{\bar\lambda})\cap\mathbb{B}_{\eta_{\bar\lambda}}(y_{\bar\lambda})\subset\Phi_k(w)+\ell_k\|w-x_{\bar\lambda}\|\mathbb{B}\;\text{ for all
}\;w\in\mathbb{B}_{\eta_{\bar\lambda }}(x_{\bar\lambda}). \end{equation} The latter yields the existence of $\lambda'>\bar\lambda$ such that $x_{\lambda'}\in\mathbb{B}(x_{\bar\lambda })$,
$y_{\lambda'}\in\Phi_k( x_{\lambda'})$, and $v_{\lambda'}\in\mathbb{B}$ for which \begin{equation*} y_{\bar\lambda}\in
y_{\lambda'}+\ell_k\|x_{\lambda'}-x_{\bar\lambda}\|v_{\lambda'}=y_{\lambda'}+\ell_k(\lambda'-\bar\lambda)\|x_1-x_0\|v_{\lambda'}. \end{equation*} Using then \eqref{equation001Purely} and
\eqref{equation003purely} implies that \begin{align*} y&=y_{\bar\lambda}+\ell_k\bar\lambda\|x_0-x_1\|v_{\bar
\lambda}=y_{\lambda'}+\ell_k(\lambda'-\bar\lambda)\|x_1-x_0\|v_{\lambda'}+\ell_k\bar\lambda\|x_0-x_1\|v_{\bar\lambda}\\
&=y_{\lambda'}+\ell_k\lambda'\|x_1-x_0\|\left(\frac{\lambda'-\bar\lambda}{\lambda'}v_{\lambda'}+\frac{\bar\lambda}{\lambda'}v_{\bar\lambda}\right)\in\Phi_k(x_{\lambda'})
+\ell_k\lambda'\|x_0-x_1\|\mathbb{B}, \end{align*} which contradicts the above choice of $\bar\lambda$. This tells us that \eqref{equation000Purely} holds, and thus $\Phi$ satisfies the
integrable Lipschitz-like property \eqref{eq_definition_Int_locPurely}.

Conversely, assume that $\Phi$ satisfies \eqref{eq_definition_Int_locPurely} with some $\ell\in\Leb^\infty(T,\R_+)$. A close look at the proof of \cite[Theorem~1.43]{m06} tells us that
\eqref{eq_definition_Int_locPurely} implies that \begin{equation}\label{Regular:Int_Lips_like_inq:foreachK}
\sup\big\{\|x^\ast\|\;\big|\;x^\ast\in\Hat{D}^\ast\Phi_t(x,y)(y^\ast)\big\}\le\ell_k\|y^\ast\|\;\text{ on }\;T_k \end{equation} whenever $x\in\mathbb{B}_{\frac{\gamma}{2(\ell_k +1)}}(\ox)$,
$y\in\mathbb{B}_{\frac{\gamma}{ 2\mu(T_k) } }(\ooy(t))\cap\Phi_t(x)$, and $y^\ast\in\Y$. Picking now any number $0<\eta< \frac{\gamma}{2(\|\ell\|_\infty +1)}$ together with measurable functions
$\x\in\mathbb{B}_\eta(\ox)$, $\y\in\mathbb{B}_\eta(\ooy)\cap\Phi(\x)$, and $\y^\ast\in\Leb^\infty(T,\Y)$ yields \begin{equation*} \|\y(t)-\ooy(t)\|\le\frac{\gamma}{2\mu(T_k)}\;\mbox{ for a.e.
}\;t\in T_k. \end{equation*} Then the application of \eqref{Regular:Int_Lips_like_inq:foreachK} verifies the integrable quasi-Lipschitzian property of $\Phi$ around $(\ox,\ooy)$.
$\h$\vspace*{0.05in}

Now we discover a remarkable phenomenon concerning Lipschitzian properties of random multifunctions, which does not have any analogs  in the deterministic framework. Namely, it is revealed that
in spaces with {\em nonatomic} measures integrable {\em quasi-Lipschitzian} and {\em local Lipschitzian} properties of random multifunctions {\em agree}. Furthermore, the measure nonatomicity is
also {\em necessary} for the fulfillment of this phenomenon as far as arbitrary multifunctions with these properties are addressed.\vspace*{-0.05in} \begin{theorem}[\bf integrably Lipschitzian
multifunctions on nonatomic spaces]\label{nonatomic:charact} Let $(T,\mathcal{A},\mu)$ be a complete finite measure spaces, and let $\Phi\colon T\times\X\tto\Y$ be a set-valued normal integrand.
Assume that the measure $\mu$ in nonatomic. Then given any $(\ox,\oy)\in\gph\Intfset{\Phi}$ and $\ooy\in\mathcal{S}_{\Phi}(\ox,\oy)$, we have that $\Phi$ is integrably quasi-Lipschitzian around
$(\ox,\ooy)$ if and only if it is integrably locally Lipschitzian around $\ox$. Conversely, the presence of at least one atom on $T$ yields the existence of a multifunction $\Phi\colon
T\times\X\tto\Y$, which is integrably quasi-Lipschitzian around $(\ox,\ooy)$ but not integrably locally Lipschitzian around $\ox$. \end{theorem}\vspace*{-0.05in} {\bf Proof}. It follows from the
equivalent descriptions of the integrably locally Lipschitzian and quasi-Lipschitzian properties in Theorem~\ref{Prop47} and Proposition~\ref{eq:regular:basic}, respectively, that the former
property implies the latter one without the nonatomicity assumption. Assume now that $\Phi$ is integrably quasi-Lipschitzian around the point in question. Arguing by contraposition, suppose that
$\Phi$ is integrably locally Lipschitzian around $\ox$, i.e., the coderivative condition \eqref{Int_Lips_like_inq_nonatomic} in Theorem~\ref{Prop47} fails. This gives us a set $A\in\mathcal{A}$
with $\mu(A)>0$ such that for all $t\in A$ there exist points $x_t\in\mathbb{B}_\eta(\ox)$, $y_t^\ast\in\Y$, $y_t\in\Phi_t(x_t)$, and $x_t^\ast\in\Hat{D}^\ast \Phi_t(x_t,y_t)(y_t^\ast)$ with
$\|x^\ast_t\|>\ell(t)\|y_t^\ast\|$. Now define the multifunction $F\colon A\tto\Y$ by \begin{align*} \gph F:=\left\{(t,x,y,x^\ast,y^\ast)\in T\times \mathbb{R}^{2(n+m)}\;\Bigg|\;\begin{array}{ll}
x\in\mathbb{B}_\eta(\ox),\;(t,x,y)\in\gph\Phi\\ (x^\ast,-y^\ast)\in\Hat N_{\gph\Phi_t}(x,y),\\\text{and }\;\|x^\ast\|>\ell(t)\|y^\ast\| \end{array}\right\}. \end{align*} Observe that $F$ is graph
measurable with nonempty (not necessary closed) values. The measurable selection theorem presented in Proposition~\ref{Prop_measurableselection} gives us measurable functions $\x,\x^\ast\colon
T\to\X$ and $\y,\y^\ast\colon T\to\Y$ such that for almost all $t\in A$ we get \begin{equation}\label{eq_01}
\x(t)\in\mathbb{B}_\eta(\ox),\;\x^\ast(t)\in\Hat{D}^\ast\Phi_t\big(\x(t),\y(t)\big)\big(\y^\ast(t)\big)\;\mbox{ and }\;\|\x^\ast(t)\|>\ell(t)\|\y^\ast(t)\|. \end{equation} The nonatomicity of
$\mu$ allows us to find a measurable set $B\subset A$ with $\mu(B)>0$ and $\int_B\|\y(t)-\ooy(t)\|d\mu\le\eta$. Hence the measurable functions $\v:=\x\1_{B}+\ox\1_{B^c}$ and
$\w:=\y\1_{B}+\ooy\1_{B^c}$ belong to the balls $\mathbb{B}_\eta(\ox)$ and $\mathbb{B}_\eta(\ooy)$, respectively. Employing now \eqref{eq_01} with taking into account
Proposition~\ref{eq:regular:basic}, we arrive at a contradiction with \eqref{Int_Lips_like_inq} and thus verify the assertion of the theorem in nonatomic spaces.

To justify the converse assertion, suppose that $(T,\mathcal{A},\mu)$ contains at least one atom, say $T_0\in\mathcal{A}$. Taking into account the representation of the measure space as the union
of purely atomic and nonatomic parts as discussed in Section~\ref{sec2a}, we assume without loss of generality that the space $T\backslash T_0$ is nonatomic; see also the proof of
Proposition~\ref{suff} for more details in similar arguments. Consider two arbitrary deterministic multifunctions $F_1,F_2\colon\X\tto\Y$ such that $F_1$ is locally Lipschitzian around some $\bar
x$, while $F_2$ is Lipschitz-like around $(\bar x,\bar y)$ with some $\oy\in F_2(\ox)$ but not being locally Lipschitzian around $\ox$. Specifically $F_2$ can be constructed as $F_2(x)=|x|$ on
$\R$ with the additional value $F(0)=1$, while $(\ox,\oy)=(0,0)$. Define now the random multifunction $\Phi\colon T\times\X\tto\Y$ by \begin{align}\label{phi} \Phi_t(x):=\left\{\begin{array}{ll}
F_1(x)&\text{ if }\;t\in T\backslash T_0,\\ F_2(x)&\text{ if }\;t\in T_0\\ \end{array}\right. \end{align} and construct a measurable selection $\ooy(t)\in\Phi_t(x)$ by
$\ooy(t):=y_1\1_{T\backslash T_0}+\bar y_1\1_{T_0}$ with some $y_1\in F_1(x)$. It follows from Theorem~\ref{Prop47} and from the first part of this theorem applied on the nonatomic part
$T\backslash T_0$ that the multifunction $\Phi$ from \eqref{phi} is integrably quasi-Lipschitzian around $(\ox,\ooy)$ when $t\in T\backslash T_0$. This property of $\Phi$ on the purely atomic
part $T_0$ follows from Theorem~\ref{Prop4700} since $F_2$ is a deterministic Lipschitz-like multifunction. However, the random multifunction $\Phi$ cannot be integrably locally Lipschitzian
around $\ox$ since $F_2$ in \eqref{phi} was chosen to be merely Lipschitz-like around the point in question. This therefore completes the proof of the theorem. $\h$\vspace*{0.05in}

It is easy to deduce from Theorem~\ref{nonatomic:charact} that the integrably Lipschitz-like and quasi-Lipschitzian properties are {\em different} in any {\em nonatomic} space. Indeed,
Theorem~\ref{nonatomic:charact} tells us that in such spaces the latter property is equivalent to its locally Lipschitzian counterpart defined by inclusion \eqref{eq_definition_Int_loc}, which
may clearly be different from the integrable Lipschitz-like property \eqref{eq_definition_Int_loc} whenever the sets $\Phi_t(x)$ are unbounded. The following simple example demonstrates the
difference between these two Lipschitzian properties of random set-valued mappings even in the case of {\em bounded} multifunctions.\vspace*{-0.05in} \begin{example} $(${\em quasi-Lipschitzian
versus Lipschitz-like random multifunctions$)$}\label{versus} Consider any nonatomic probability space $(T,\mathcal{A},\mu)$ and the constant mapping $\Phi\colon T\times\R\tto\R_+$ defined by
$\Phi(t,x):=[0,\sqrt{|x|}+1]$. Then we have $\Intf{\Phi}(x)=[0,\sqrt{|x|}+1]$ with $(\ox,\oy)=(0,0)\in\gph\Intfset{\Phi}$ and the constant function $\ooy(t)=0$ belongs to
$\mathcal{S}_{\Phi}(\ox,\oy)$ from \eqref{mapping:SPHI}. It is easy to see that $\Intf{\Phi}(x)$ is Lipschitz-like around $(\ox,\oy)$. However, $\Phi$ cannot be quasi-Lipschitzian around this
point, because otherwise it follows from Theorem~\ref{Prop47} and \ref{nonatomic:charact} that $\Phi$ satisfies \eqref{eq_definition_Int_loc}. The latter leads us to a contradiction, since it
would imply, in particular, that the function $\ph(x):= \sqrt{|x|}$ is locally Lipschitzian around $\ox=0$. \end{example}

Similarly to the deterministic case, the integrable Lipschitz-like property brings us to conclusions about {\em Lipschitz stability} of feasible solution sets in stochastic programming, while the
integrable quasi-Lipschitzian property plays a crucial role in deriving the most useful {\em pointwise Leibniz-type rules} for coderivatives and second-order subdifferentials that are established
in the subsequent sections.\vspace*{0.03in}

The final result of this section provides efficient conditions ensuring the integral quasi-Lipschitzian property of set-valued normal integrands defined on general measure spaces that are
important, in particular, to check the qualification conditions in the calculus rules derived below.\vspace*{-0.05in} \begin{proposition}[\bf sufficient conditions for integrably
quasi-Lipschitzian multifunctions]\label{suff} Let $(T,\mathcal{A},\mu)$ be a complete finite measure space, and let $\Phi\colon T\times\X\tto\Y$ be a set-valued normal integrand. Take
$(\ox,\oy)\in\gph\Intfset{\Phi}$ and $\ooy\in\mathcal{S}_{\Phi}(\ox,\oy)$ and suppose that there exist $\epsilon,\ell_{pa}>0$ as well as $\gamma_{pa}>0$ and $\ell_{na}\in\Leb^1(T,\mathbb{R_+})$
such that the following inclusions hold for all $x,x'\in\mathbb{B}_\epsilon(\ox)$: \begin{equation}\label{inclusion01} \begin{array}{ll}
\Phi_t(x)\cap\mathbb{B}_{\frac{\gamma_{pa}}{\mu(T_k)}}\big(\oy(t)\big)\subset\Phi_t(x')+\ell_{pa}\|x-x'\|\mathbb{B}\;\text{ if }\;t\in T^k_{pa},\\
\Phi_t(x)\subset\Phi_t(x')+\ell_{na}(t)\|x-x'\|\mathbb{B}\;\text{ if }\;t\in T_{na}, \end{array} \end{equation} where $T_{pa}$ and $T_{na}$ form a disjoint decomposition of $T$ into purely atomic
and nonatomic parts with $\{T_{pa}^k\}_{k\in\N}$ being a countable disjoint family of atoms. Then $\Phi$ is quasi-Lipschitzian around $(\ox,\ooy)$. \end{proposition}\vspace*{-0.05in} {\bf Proof}.
Let $\mu_{pa}(\cdot)$ and $\mu_{na}(\cdot)$ be the purely atomic and nonatomic part of $\mu$, i.e., $\mu_{pa}(\cdot):=\mu(\cdot\cap T_{pa})$ and $\mu_{na}(\cdot):=\mu(\cdot\cap T_{na})$, ,
respectively. Using the second inclusion in \eqref{inclusion01}, we apply Theorems~\ref{Prop47} and \ref{nonatomic:charact} to the multifunction $\Phi$ on the nonatomic measure space
$(T_{na},\mathcal{A},\mu_{na})$ and thus find $\eta_{na}>0$ and $\ell_{na}\in\Leb^1(T,\R_+)$ such that \eqref{Int_Lips_like_inq} holds. On the other hand, the first inclusion in
\eqref{inclusion01} allows us to apply Theorem~\ref{Prop4700} to $\Phi$ on the purely atomic measure space $(T_{pa},\mathcal{A},\mu_{pa})$ and hence get $\eta_{pa}>0$ and $\ell_{pa}\in
\Leb^1(T,\R_+)$ such that \eqref{Int_Lips_like_inq} holds. Now we claim that $\Phi$ satisfies \eqref{Int_Lips_like_inq} with $\eta:=\min\{\eta_{na},\eta_{pa}\}$ and
$\ell(t):=\ell_{na}(t)\1_{T_{na}}(t)+\ell_{pa}(t)\1_{T_{pa}}(t)$ on the entire measure space $(T,\mathcal{A},\mu)$. Indeed, picking any $\x\in\mathbb{B}_\eta(\ox)$,
$\y\in\mathbb{B}_\eta(\ooy)\cap\Phi(\x)$, and $\y^\ast\in\Leb^\infty(T,\Y)$ leads us to the estimate \begin{equation*}
\int_{T}\|\y(t)-\ooy(t)\|d\mu_{na}+\int_{T}\|\y(t)-\ooy(t)\|d\mu_{pa}=\int_{T}\|\y(t)-\ooy(t)\|d\mu\le\eta, \end{equation*} which yields \eqref{Int_Lips_like_inq} for the selected triple
$(\x,\y,\y^\ast)$ and thus completes the proof of the proposition. $\h$\vspace*{-0.2in}

\section{Coderivative Leibniz Rules for Expected-Integral Mappings}\label{sec4}\sce\vspace*{-0.1in}

In this section we establish several calculus rules to evaluate coderivatives of the expected-integral multifunctions $\Intfset{\Phi}\colon\X\tto\Y$ given by
\begin{equation}\label{def:set-valued:Exp*}
\Intfset{\Phi}(x):=\int_{T}\Phi_t(x)d\mu\;\mbox{ for all }\;x\in\X,
\end{equation}
where $\Phi\colon T\times\X\tto\Y$ is a set-valued normal integrand defined on a complete finite measure space $(T,\mathcal{A},\mu)$. Various results, unified under the name of {\em coderivative Leibniz
rules}, are obtained to evaluate the regular and limiting coderivatives of \eqref{def:set-valued:Exp*} via the corresponding coderivatives of the integrand multifunctions $\Phi_t$. The most efficient
and useful {\em pointwise} rules will be obtained to evaluate the limiting coderivative of $\Intfset{\Phi}$ at the given point $(\ox,\oy)$ in terms of the limiting coderivative of $\Phi_t$ under the
integrable {\em quasi-Lipschitzian} property of the integrand multifunctions around the corresponding points.\vspace*{0.05in}

Throughout this section we assume that at the given point of interest $\ox\in\dom\Intfset{\Phi}$ there exist a positive number $\rho$ and a function $\kappa\in\Leb^1(T,\R_+)$ such that the following
conditions are satisfied:
\begin{equation}\label{convex_cond_set}
\begin{aligned}
\Phi_t(x)&\text{ is convex for all }\;x\in\mathbb{B}_\rho(\ox)\;\text{ and a.e. }\;t\in T_{na},\\
\Phi_t(x)&\subset\kappa(t)\mathbb{B}\;\text{ for all }\;x\in\mathbb{B}_\rho(\ox)\;\text{ and a.e. }\;t\in T,
\end{aligned}
\end{equation}
where the second condition is known as {\em integrable boundedness}.\vspace*{0.05in}

We start with deriving {\em sequential} Leibniz-type rules to estimate the {\em regular} coderivative of $\Intfset{\Phi}$ via sequences of regular coderivatives of $\Phi_t$ at points nearby.
These results are of the same flavor as the sequential subdifferential Leibniz rules for expected-integral functionals \eqref{eif} that are presented in Theorems~\ref{theoremsubdiferential} and
\ref{theorem_main_fuzzy_sub} and are used in the proofs below. Since the proofs of the following two propositions are similar to each other, we prove only the second one.\vspace*{-0.05in}
\begin{proposition}[\bf sequential coderivative Leibniz rule, I]\label{main_fuzzy_cod2} Let $\Phi\colon T\times\X\tto\Y$ be a set-valued normal integrand defined on a complete finite measure
space $(T,\mathcal{A},\mu)$, let $\ox\in\dom\Intfset{\Phi}$ satisfy the conditions in \eqref{convex_cond_set}, and let $\ooy\in\mathcal{S}_{\Phi}(\ox,\oy)$ be taken from \eqref{mapping:SPHI}.
Then for any fixed $p,q\in(1,\infty)$ with $1/p+1/q=1$ there exist sequences $\{x_k\}\subset\X$, $\{\x_k\}\subset\Leb^p({T},\X)$, $\{\x_k^*\}\subset{\Leb}^q({T},\X)$,
$\{\y_k\}\subset\Leb^1(T,\Y)$, and $\{\y_k^\ast\}\subset\Leb^\infty(T,\Y)$ such that we have the assertions:\vspace*{-0.05in} \begin{enumerate}[label=\alph*),ref=\alph*)] \item[\bf(i)]
$\x_k^*(t)\in\Hat{D}^\ast\Phi_t\big(\x_k(t),\y_k(t)\big)\big(\y_k^\ast(t)\big)$ for a.e.\ $t\in T$ and all $k\in\N$. \item[\bf(ii)] $\|\ox-x_k\|\to 0$, $\|\ox-\x_k\|_p \to 0$, and
$\|\ooy-\y_k\|_1\to 0$ as $k\to\infty$. \item[\bf(iii)] $\| \y_k^\ast-\ooy^\ast\|_\infty\to 0$, $\|\x_k^*\|_q\|\x_k-x_k\|_p\to 0$, and $\disp\int_T x_k^*(t)d\mu\to\ox^\ast$ as $k\to\infty$.
\end{enumerate} \end{proposition}\vspace*{-0.13in} \begin{proposition}[\bf sequential coderivative Leibniz rule, II]\label{main_fuzzy_cod} Let
$\ox^\ast\in\Hat{D}^\ast\Intfset{\Phi}(\ox,\oy)(\oy^\ast)$ in the setting of Proposition~{\rm\ref{main_fuzzy_cod}}. Then there exist sequences $\{x_k\}\subset\X$,
$\{\x_k\}\subset\Leb^\infty({T},\X)$, $\{\x_k^\ast\}\subset{\Leb}^1({T},\X)$, $\{\y_k\}\subset\Leb^1(T,\Y)$, and $\{\y_k^\ast\}\subset\Leb^\infty(T,\Y)$ such that the following assertions hold:
\begin{enumerate}[label=\alph*),ref=\alph*)] \item[\bf(i)] $\x_k^*(t)\in\Hat{D}^\ast\Phi_t\big(\x_k(t),\y_k(t)\big)\big(\y_k^\ast(t)\big)$ for a.e.\ $t\in T$ and all $k\in\N$. \item[\bf(ii)]
$\|\ox-x_k\|\to 0$, $\|\ox-\x_k\|_\infty\to 0$, $\disp\int_T\|\ooy(t)-\y_k(t)\|d\mu\to 0$, and $\|\y_k^\ast-\oy^\ast\|_\infty\to 0$ as $k\to\infty$. \item[\bf(iii)]
$\disp\int_T\|\x_k^*(t)\|\cdot\|\x_k(t)-x_k\|d\mu\to 0$ and $\disp\int_T\x_k^*(t) d\mu\to\ox^\ast$. \end{enumerate} \end{proposition}\vspace*{-0.05in} {\bf Proof}. Considering the
extended-real-valued function $\varphi(t,v,w):=\delta_{\gph\Phi_t}(u,w)$, observe that $\varphi$ is a normal integrand. Furthermore, it is easy to see that the imposed assumptions in
\eqref{convex_cond_set} yield the fulfillment of \eqref{convex_cond} for $\ph$. Picking any $\ox^\ast\in\Hat{D}^\ast\Intfset{\Phi}( \ox,\oy)(\oy^\ast)$ gives us by regular coderivative definition
\eqref{rcod} that $(\ox^\ast,\oy^\ast)\in\Hat{\partial}\delta_{\gph\Intfset{\Phi}}(\ox,\oy)$. Now we use the smooth variational description of regular subgradients taken from
\cite[Theorem~1.88(i)]{m06}, which tells us that there exist $\eta>0$ and a function $\vt\colon\B_\eta(\ox,\oy)\to\X\times\Y$ that is Fr\'echet differentiable at $(\ox,\oy)$ with
$\nabla\vt(\ox,\oy)=(\ox^\ast,\oy^\ast)$ and such that the difference $\delta_{\gph\Intfset{\Phi}}-\vt$ attains its minimum at $(\ox,\oy)$ on $\mathbb{B}_{\eta}(\ox,\oy)$. Observe that
\begin{align*} \delta_{\gph\Intfset{\Phi}}\bigg(u,\int_T\ooy(t)d\mu\bigg)&=\int_T \varphi_t\big(u,\ooy(t)\big)d\mu\;\mbox{ and}\\
\delta_{\gph\Intfset{\Phi}}\bigg(u,\int_T\w(t)d\mu\bigg)&\le\int_T\varphi_t\big(u,\w(t)\big)d\mu\;\text{ for all }\;(u,\w)\in\X\times\Leb^1(T,\Y). \end{align*} Consider further the
extended-real-valued function \begin{equation*} \X\times\Leb^1(T,\Y)\ni(u,\w)\to\Intf{\varphi}(u,\w)-\psi(u,\w)\;\mbox{ with }\;\psi(u,\w):=\vt\bigg(u,\int_T\w(t)d\mu\bigg), \end{equation*} which
attains a local minimum at $(\ox,\ooy)$. Hence we have by the elementary version of the subdifferential Fermat rule (see, e.g., \cite[Proposition~1.114]{m06}) that
$(0,0)\in\Hat\partial(\Intf{\varphi}-\psi)(\ox,\ooy)$. Observe that $\psi$ is clearly Fr\'echet differentiable at $(\ox,\ooy)$ with $\nabla\psi(\ox,\ooy)=(\ox^\ast,\ooy^\ast)$, where $\ooy^\ast$
is the constant function $\ooy^\ast(t):=\oy^\ast\1_T(t)$. Employing now the sum rule from \cite[Proposition~1.107]{m06}) yields $(\ox^\ast,\ooy^\ast)\in\Hat{\partial}\Intf{\varphi}(\ox,\ooy)$.
Furthermore, it follows from \eqref{convex_cond_set} that the function \begin{equation*} t\to\inf_{\mathbb{B}_{\rho}(\ox)\times\Y}\big\{\varphi_t(\cdot,\cdot)-\la\ooy^\ast(t),\cdot\ra\big\}
\end{equation*} is integrable on $T$. Finally, the application of Theorem~\ref{theorem_main_fuzzy_sub} completes the proof of the proposition. $\h$\vspace*{0.05in}

The subsequent results of this section establish {\em pointwise} versions of coderivative Leibniz rules under the {\em integrable quasi-Lipschitzian} property of set-valued integrands $\Phi_t$.
The first theorem gives us an upper estimate of the {\em regular coderivative} of $\Intfset{\Phi}$ via the integral of the limiting coderivative of $\Phi_t$.\vspace*{-0.05in} \begin{theorem}[\bf
pointwise estimate for regular coderivatives of expected-integral mappings]\label{Theo_lim_reg_Cod} Let $\Phi\colon T\times\X\tto\Y$ be a set-valued normal integrand on a complete finite measure
space $(T,\mathcal{A},\mu)$, let $\ox\in\dom\Intfset{\Phi}$ satisfy the conditions in \eqref{convex_cond_set}, and let $\ooy\in\mathcal{S}_{\Phi}(\ox,\oy)$ for some $\oy\in\Intfset{\Phi}(\ox)$.
Assume in addition that the integrable quasi-Lipschitzian property holds for $\Phi$ around $(\ox,\ooy)$. Then we have the inclusion \begin{equation}\label{reg-point}
\Hat{D}^\ast\Intfset{\Phi}(\ox,\oy)(y^\ast)\subset\int_T{D}^\ast\Phi_t\big(\ox,\ooy(t)\big)(y^\ast)d\mu\;\mbox{ for all }\;y^\ast\in\Y. \end{equation} \end{theorem}\vspace*{-0.05in} {\bf Proof}.
Pick any $\ox^\ast\in\Hat{D}^\ast\Intfset{\Phi}(\ox,\oy)(\oy^\ast)$. Proposition~\ref{main_fuzzy_cod} allows us to find sequences $\{x_k\}\subset\X$, $\{\x_k\}\subset\Leb^\infty({T},\X)$,
$\{\x_k^\ast\}\subset{\Leb}^1({T},\X)$, $\{\y_k\}\subset\Leb^1(T,\Y)$, and $\{\y_k^\ast\}\subset\Leb^\infty(T,\Y)$ satisfying the conditions
$\x_k^*(t)\in\Hat{D}^\ast\Phi_t(\x_k(t),\y_k(t))(\y_k^\ast(t))$ for a.e.\ $t\in T$, $\|\ox-\x_k\|_\infty\to 0$, $\|\y_k^\ast- \oy^\ast\|_\infty\to 0 $, as well as \begin{equation*}
\int_T\big\|\ooy(t)-\y_k(t)\big\|d\mu\to 0\;\mbox{ and }\;\bigg\|\int_T\x_k^*(t)d\mu-\ox^\ast\bigg\|\to 0\;\mbox{ when }\;k\to\infty. \end{equation*} Passing to a subsequence if necessary, we get
that that $(\x_k(t),\y_k(t),\y_k^\ast(t))\to(\ox,\ooy(t),\oy^\ast)$ as $k\to\infty$ for a.e.\ $t\in T$,  which implies by using the limiting coderivative representation \eqref{Lim_repr_cod} that
\begin{equation}\label{incl_01} \Limsup_{k\to\infty}\big\{\x_k^*(t)\big\}\subset D^\ast\Phi_t\big(\ox,\ooy(t)\big)(\oy^\ast)\;\mbox{ for a.e. }\;t\in T. \end{equation} Employing now the imposed
integrable quasi-Lipschitzian property of $\Phi$ around $(\ox,\ooy)$ yields \begin{equation*} \|\x_k^\ast(t)\|\le\ell(t)\|\y_k^\ast(t)\|\le\ell(t)M\;\text{ for a.e. }\;t\in T\;\mbox{ and large
}\;k\in\N, \end{equation*} where $M:=\sup\{\|\y_k^\ast\|_\infty\mbox{ over }k\in\N\}$. Then Fatou's lemma for multifunctions taken from \cite[Corollary~4.1]{bs} and being combined with
\eqref{incl_01} tells us that \begin{align*}
\ox^\ast\in\Limsup_{k\to\infty}\left\{\displaystyle\int_{{T}}\x_k^*(t)d\mu\right\}&\subset\displaystyle\int_{{T}}\Limsup_{k\to\infty}\big\{\x_k^*(t)\big\}d\mu\\
&\subset\displaystyle\int_{{T}}D^\ast\Phi_t\big(\ox,\ooy(t)\big)(y^\ast)d\mu \end{align*} for all $y^*\in\Y$. This shows that
$\ox^\ast\in\disp\int_{{T}}D^\ast\Phi_t\big(\ox,\ooy(t)\big)(y^\ast)d\mu$ and thus verifies \eqref{reg-point}. $\h$\vspace*{0.05in}

To proceed further with deriving an efficient pointwise upper estimate of the {\em limiting coderivative} of  $\Intfset{\Phi}$ via the limiting coderivative of the integrand $\Phi_t$, we need to
invoke an additional property of the set-valued mapping $\mathcal{S}_\Phi$ defined in \eqref{mapping:SPHI}. This property of multifunctions, known as {\em inner semicompactness}, is formulated
and discussed in \cite[Definition~1.63(ii)]{m06}. The reader can find in \cite[Definition~1.63(i)]{m06} a parallel {\em inner semicontinuity} property of multifunctions, which could also be used
to establish complemented coderivative Leibniz rules, while we are not going to pursue this aim in this paper.

Recall that $\mathcal{S}_\Phi$ is {\em inner semicompact} at $(\ox,\oy)$ if for every sequence $(x_k,y_k)\to(\ox,\oy)$ there exists a sequence $\y_k\in\mathcal{S}_\Phi(x_k,y_k)$ that contains an
$\Leb^1({T},\Y)$-norm convergent subsequence as $k\to\infty$.\vspace*{-0.05in}
\begin{theorem}[\bf coderivative Leibniz rule for integrably quasi-Lipschitzian
multifunctions]\label{Theo_Basic_coder} Let $\Phi\colon T\times\X\tto\Y$ be a set-valued normal integrand on a complete finite measure space $(T,\mathcal{A},\mu)$, and let
$\ox\in\dom\Intfset{\Phi}$ satisfy the conditions in \eqref{convex_cond_set}. Take $\oy\in\Intfset{\Phi}(\ox)$ such that the mapping $\mathcal{S}_\Phi$ is inner semicompact at $(\ox,\oy)$ and
assume in addition that $\Phi$ is integrably quasi-Lipschitzian around $(\ox,\ooy)$ for all $\ooy\in\mathcal{S}_\Phi(\ox,\oy)$. Then we have the limiting coderivative Leibniz rule
\begin{equation}\label{lim-point} {D}^\ast\Intfset{\Phi}(\ox,\oy)(y^\ast)\subset\bigcup\limits_{\ooy\in\mathcal{S}_{\Phi}(\ox,\oy)}\int_T{D}^\ast\Phi_t\big(\ox,\ooy(t)\big)(y^\ast)d\mu\;\mbox{
whenever }\;y^\ast\in\Y. \end{equation} \end{theorem}\vspace*{-0.05in} {\bf Proof}. Pick any $\ox^\ast\in{D}^\ast\Intfset{\Phi}(\ox,\oy)(\oy^\ast)$. It follows from representation
\eqref{Lim_repr_cod} of the limiting coderivative that there exist the convergent quadruples $(x_k,y_k,x_k^\ast,y_k^\ast)\to(\ox,\oy,\ox^\ast,\oy^\ast)$ with
$x_k^\ast\in\Hat{D}^\ast\Intfset{\Phi}(x_k,y_k) (y_k^\ast)$ for all $k\in \N$. Employing the imposed inner semicompactness of $\mathcal{S}_\Phi$ and passing to a subsequence if necessary give us
$\y_k\in\mathcal{S}_\Phi(x_k,y_k)$ such that $\y_k \to \ooy$ in the norm topology of $\Leb^1({T},\Y)$. Remembering that the graph of $\Phi_t$ is closed yields $\ooy\in\mathcal{S}_\Phi(\ox,\oy)$.
Applying further Proposition~\ref{main_fuzzy_cod} to each quadruple $(x_k,y_k,x_k^\ast,y_k^\ast)$ and employing the diagonal process ensure the existence of sequences
$\{\x_k\}\subset\Leb^\infty({T},\X)$, $\{\x_k^\ast\}\subset{\Leb}^1({T},\X)$, $\{\y_k\}\in\Leb^1(T,\Y)$, and $\{\y_k^\ast\}\in\Leb^\infty(T,\Y)$ satisfying the following conditions:
$\x_k^*(t)\in\Hat{D}^\ast\Phi_t(\x_k(t),\y_k(t))(\y_k^\ast(t))$ for a.e.\ $t\in T$ and all $k\in\N$ together with the norm convergence $\|\ox-\x_k\|_\infty\to 0$, $\|\ooy-\y_k\|_1\to 0$, and
$\|\y_k^\ast-\oy^\ast\|_\infty\to 0$ as well as \begin{equation*} \bigg\|\int_T\x_k^*(t)d\mu-\ox^\ast\bigg\|\to 0\;\mbox{ and
}\;\big(\x_k(t),\y_k(t),\y_k^\ast(t)\big)\to\big(\ox,\ooy(t),\oy^\ast\big)\;\mbox{ for a.e. }\;t\in T\;\mbox{ as }\;k\to\infty. \end{equation*} Furthermore, we have the limiting inclusion
\eqref{incl_01} as obtained in the proof of Theorem~\ref{Theo_lim_reg_Cod}. Using now the assumed integrable quasi-Lipschitzian property of $\Phi$ around $(\ox,\ooy)$ and remembering the
${\Leb}^\infty$-norm convergence $\y_k\to\ooy$ as $k\to\infty$ allow us to find $\ell\in\Leb^1(T,\R_+)$ and $k_0\in\N$ such that \begin{equation*}
\|\x_k^\ast(t)\|\le\ell(t)\|\y_k^\ast(t)\|\le\ell(t)M\;\mbox{ for almost all }\;t\in T\;\mbox{ and all }\;k\ge k_0, \end{equation*} where $M=\sup\{\|\y_k\|_\infty\;\mbox{over}\;k\in\N\}$. Hence
the functions $\x_k$ are uniformly integrable on $T$. It follows from \eqref{incl_01} and the aforementioned Fatou's lemma for multifunctions that for all $y^*\in\Y$ we get \begin{align*}
\ox^\ast\in\displaystyle\int_{{T}}D^\ast\Phi_t\big(\ox,\ooy(t)\big)(y^\ast)d\mu\subset\bigcup\limits_{\ooy\in\mathcal{S}_\Phi(\ox,\oy)}\int_T{D }^\ast\Phi_t\big(\ox,\ooy(t)\big)(y^\ast)d\mu,
\end{align*} which verifies \eqref{lim-point} and completes the proof of the theorem. $\h$\vspace*{0.05in}

The following result corresponds to an alternative Leibniz rule for coderivative of expected multifunctions when the mapping $\mathcal{S}$ is not inner semicompact at the reference point, but the
mapping $\Phi$ is integrably locally Lipschitzian.\vspace*{-0.05in}

\begin{theorem}[\bf coderivative Leibniz rule for integrably Lipschitzian multifunctions]\label{Theo_Basic_coder2} Let $\Phi\colon T\times\X\tto\Y$ be a set-valued normal integrand on a complete
finite measure space $(T,\mathcal{A},\mu)$, and let $\ox\in\dom\Intfset{\Phi}$. Suppose that  $\Phi$ is integrably locally Lipschitzian around $\ox$. Then for every $\oy  \in \Intfset{\Phi}(\ox)$
we have the limiting coderivative Leibniz rule \begin{equation}\label{lim-point2} {D}^\ast\Intfset{\Phi}(\ox,\oy)(y^\ast)\subset\int_T{D}^\ast\Phi_t\big(\ox,\Phi_t(\ox)\big)(y^\ast)d\mu\;\mbox{
whenever }\;y^\ast\in\Y, \end{equation} where ${D}^\ast\Phi_t\big(\ox,\Phi_t(\ox)\big)(y^\ast):=\bigcup_{ y\in\Phi_t(\ox)}{D}^\ast\Phi_t\big(\ox,y)(y^\ast)$. \end{theorem} \begin{proof} Following
the arguments of Theorem \ref{Theo_Basic_coder}, we have the existence of  sequences $\{\x_k\}\subset\Leb^\infty({T},\X)$, $\{\x_k^\ast\}\subset{\Leb}^1({T},\X)$, $\{\y_k\}\in\Leb^1(T,\Y)$, and
$\{\y_k^\ast\}\in\Leb^\infty(T,\Y)$ satisfying the following conditions: $\x_k^*(t)\in\Hat{D}^\ast\Phi_t(\x_k(t),\y_k(t))(\y_k^\ast(t))$ for a.e.\ $t\in T$ and all $k\in\N$ together with the norm
convergence $\|\ox-\x_k\|_\infty\to 0$, and $\|\y_k^\ast-\oy^\ast\|_\infty\to 0$ as well as $\|\int_T\x_k^*(t)d\mu-\ox^\ast\|\to 0$ and $\|\x_k^\ast(t)\|\le \ell(t)M$ for a.e. $t\in T$. Now, by
Fatou's lemma for multifunctions we get that

\begin{align*} x^\ast\in\Limsup_{k\to\infty}\left\{\displaystyle\int_{{T}}\x_k^*(t)d\mu\right\}&\subset\displaystyle\int_{{T}}\Limsup_{k\to\infty}\big\{\x_k^*(t)\big\}d\mu. \end{align*} To
conclude the proof, let us show the following inclusion $\Limsup_{k\to\infty}\{\x_k^*(t)\} \subset{D}^\ast\Phi_t\big(\ox,\Phi_t(\ox)\big)(y^\ast)$ for almost all $t\in T$. Indeed, consider a set
of full measure $\hat{T}$, where $\x_k(t) \to \ox$ and $\y_k(t) \to \oy$. Fix $t\in \hat{T}$ and $u^\ast \in  \Limsup_{k\to\infty}\{\x_k^*(t)\} $. Hence, by definition there exists a subsequence
$x^\ast_{k_j}(t) \to u^\ast$. Since $\|\y_{k_j}(t)\| \leq \kappa(t)$ (recall \eqref{convex_cond_set}), we can take assume that  $\y_{k_j}  \to y$ for some $y\in \Phi_t(\ox)$ (recall that the
graph of $\Phi_t$ is closed). Therefore, by the limiting coderivative representation \eqref{Lim_repr_cod}, we get that $x^\ast \in \limsup_{j\to \infty}
\Hat{D}^\ast\Phi_t(\x_{k_j}(t),\y_{k_j}(t))(\y_{k_j}^\ast(t))\subset{D}^\ast\Phi_t\big(\ox,y)\big)(y^\ast)$, and that concludes the proof. \end{proof}\vspace*{-0.05in}

Next we present a simple consequence of Theorem~\ref{Theo_Basic_coder} that is used in what follows.\vspace*{-0.05in}

\begin{corollary}[\bf coderivative Leibniz rule for single-valued expected-integral mappings]\label{cor_Basic_coder} Under the general assumptions of Theorem~{\rm\ref{Theo_Basic_coder}}, suppose
in addition that $\Intfset{\Phi}(\ox)$ is a singleton, and that $\Phi$ is integrably locally Lipschitzian around $\ox$. Then the mapping $\mathcal{S}_\Phi$ is inner semicompact at
$(\ox,\Intfset{\Phi}(\ox))$ and we have the coderivative Leibniz rule
 \begin{equation}\label{cor_Basic_coder_eq_01}
{D}^\ast\Intfset{\Phi}(\ox)(y^\ast)\subset\int_T{D}^\ast\Phi_t\big(\ox\big)(y^\ast)d\mu\;\mbox{ for all }\;y^*\in\Y. \end{equation} \end{corollary}\vspace*{-0.05in} {\bf Proof}. Denote
$\oy:=\Intf{\Phi}(\ox)$ and take a measurable selection $\ooy(t)\in\Phi_t(\ox)$ for a.e.\ $t\in T$. To verify that $\mathcal{S}_\Phi$ is inner semicompact at $(\ox,\oy)$, let $(x_k,y_k)\to
(\ox,\oy)$ as $k\to\infty$, and let $\y_k\in\mathcal{S}_\Phi(x_k,y_k)$ for all $k\in\N$. Employing the integrable local Lipschitzian property \eqref{eq_definition_Int_loc} of $\Phi$ around $\ox$
tells us that \begin{align*} \|\y_k(t)-\ooy(t)\|\le\ell(t)\|x_k-\ox\|\;\text{ for a.e. }\;t\in T\;\mbox{ and all large }\;k\in\N. \end{align*} Then it follows from Lebesgue's dominated
convergence theorem that $\y_k\to\ooy$ in $\Leb^1(T,\Y)$ as $k\to\infty$, which verifies the inner semicompactness property of $\mathcal{S}_\Phi$ at $(\ox,\oy)$. Applying finally
Theorem~\ref{Theo_Basic_coder}, we arrive at \eqref{cor_Basic_coder_eq_01} and thus complete the proof of the corollary. $\h$\vspace*{0.05in}

As we see, the previous versions of the coderivative Leibniz rule provided just {\em upper estimates} of the regular and limiting coderivatives of the expected-integral mappings. Although
calculus rules as inclusions of this type suffice for many applications, it is important to establish efficient conditions ensuring the fulfillment of such results as {\em equalities}. The
following theorem does the job. We say that a set-valued mapping $F$ is {\em coderivative regular} at $(\ox,\oy)\in\gph F$ for $y^\ast\in \Y$ if $D^\ast F(\ox,\oy)(y^\ast)=\Hat{D}^\ast
F(\ox,\oy)(y^\ast)$.

\begin{theorem}[\bf coderivative Leibniz rule as equality]\label{codL-eq} In the setting of Corollary~{\rm\ref{cor_Basic_coder}}, {\em suppose that for a.e.\ $t\in T$ we have that
$\Phi_t(\ox)=\{\ooy(t)\}$ and that $\Phi_t$ is coderivative regular at $(\ox,\ooy(t))$ for $y^\ast\in \Y$. Then} \begin{equation}\label{cor_Basic_coder_eq_01:equality} \Hat
{D}^\ast\Intfset{\Phi}(\ox)(y^\ast)={D}^\ast\Intfset{\Phi}(\ox)(y^\ast)=\int_T{D}^\ast\Phi_t\big(\ox\big)(y^\ast)d\mu. \end{equation} \end{theorem} \vspace*{-0.05in} \begin{proof}  Taking into
account the inclusion of Theorem~{\rm\ref{Theo_Basic_coder}}, it is sufficient to show that \begin{equation}\label{cod-eq1} \int_T{D}^\ast\Phi_t\big(\ox,\ooy(t)\big)(y^\ast)d\mu\subset
\Hat{D}^\ast\Intfset{\Phi}(\ox)(y^\ast). \end{equation} First we check that for all $(u,v)\in\X\times \Y$ there exists $\v\in\Leb^1(T,\Y)$ with $v=\int_{T}\v(t)d\mu$ such that
\begin{align}\label{ineqsubd} \int_{T}\overline{\co}d_{\gph \Phi_t }(\ox,\ooy(t))(u,\v(t))d\mu\le d_{\gph\Intfset{\Phi}}(\ox,\oy)(u,v), \end{align} where $d_{\gph\Intfset{\Phi}}$ denotes the
subderivative \eqref{subderivative} of the indicator function of $\gph\Intfset{\Phi}$, and where $\overline{\co}d_{\gph\Phi_t}$ is the convex closure of this subderivative. This claim is obvious
if right-hand side of \eqref{ineqsubd} is $\infty$. Otherwise, suppose that $d_{\gph\Intfset{\Phi}}(\ox,\oy)(u,v)<\infty$, we find sequences $s_k\dn 0$ and $(u_k,v_k)\to 0$ as $k\to\infty$ such
that \begin{align*} d_{\gph\Intfset{\Phi}}(\ox,\oy)(u,v)=\lim\limits_{k\to\infty}\frac{\delta_{\gph\Intfset{\Phi}}(\ox + s_k u_k,\oy_k+s_k v_k)}{s_k}.\end{align*} The above inequality implies
that whenever $k\in\N$ is sufficiently large there exists $\w_k(t)\in\Phi_t(\ox + s_ku_k)$ for a.e. $t\in T$ such that $\int_T\w_k(t)d\mu =\oy + s_k v_k $. By the integrable Lipschitz continuity
of $\Phi$, it can be assumed that for a.e. $t \in T$ and all $k\in\N$ we have the representation \begin{align*}\w_k(t)=\ooy(t)+s_k\v_{k}(t)\;\text{ with some measurable function }\;\v_k(t)\in\|
u_k\|\ell(t) \mathbb{B}, \end{align*} which tells us that $\|\v_{ k}(t)\|\le\max\big\{\|u_k\|\;\big|\;k\in\N\}\ell(t)$. This implies that a subsequence of $\{\v_k\}$ converges weakly to some $\v
\in\Leb^1(T,\Y)$. Thus $v_k=\int_{T}\v_{k}(t)d\mu(t)\to \int_{T} \v(t) d\mu $ along this subsequence, and therefore we arrive at $v=\int_{T} \v(t) d\mu$.

Consider further the integrand $\Psi\colon T\times[0,\infty]\times\X\times\Y\to\Rex$ defined by \begin{align} \Psi(t,r,a,b):= \left\{ \begin{array}{cc}\disp\frac{ \delta_{\gph \Phi_t } ( \ox + r
a ,\ooy(t)+rb)}{r} &\text{ if } r>0,\\ &\\ \overline{\co} d_{\gph \Phi_t } (\ox,\ooy(t)) (a,b )  &\text{ if } r=0. \end{array}\right. \end{align} It is easy to see that $\Psi(t,\cdot)$ is lower
semicontinuous for all $t \in T$, and also that for every $(t,r,a) \in T_{na} \times [0,\infty) \times \X $ the function $\Psi(t,r,a,\cdot)$ is convex. Applying \cite[Theorem~2.1]{bal} on $T_{na}
$, we get that \begin{align*}\int_{T_{pa}} \overline{\co} d_{\gph
\Phi_t}\big(\ox,\ooy(t)\big)\big(u,\v(t)\big)d\mu&=\int_{T_{na}}\Psi\big(t,0,u,\v(t)\big)d\mu\le\liminf\limits_{k\to\infty}\int_{T_{na}}\Psi\big(t,s_k,u_k,\v_k(t)\big)d\mu=0. \end{align*} On the
other hand, noticing that $\v_{k}$ converges pointwise to $\v$ on $T_{pa}$, we deduce from Fatou's lemma and the lower semicontinuity of $\Psi$ that
\begin{align*}\int_{T_{pa}}\overline{\co}d_{\gph
\Phi_t}\big(\ox,\ooy(t)\big)\big(u,\v(t)\big)d\mu&=\int_{T_{pa}}\Psi\big(t,0,u,\v(t)\big)d\mu\le\int_{T_{na}}\liminf\limits_{k\to\infty}\Psi\big(t,s_k,u_k,\v_k(t)\big)\\&\le\liminf\limits_{k\to
\infty}\int_{T_{na}}\Psi\big(t,s_k,u_k,\v_k(t)\big)d\mu= 0,\end{align*} which therefore verifies the estimate in \eqref{ineqsubd}. To proceed, pick
$x^\ast\in\int_T{D}^\ast\Phi_t\big(\ox,\ooy(t)\big)(y^\ast)d\mu$ and then find, by using the assumed graphical regularity of the integrand $\Phi_t$, a measurable selection $\x(t)\in
{D}^\ast\Phi_t\big(\ox,\ooy(t)\big)(y^\ast)=\Hat{D}^\ast\Phi_t\big(\ox,\ooy(t)\big)(y^\ast)$ such that $x^\ast=\int_T \x^\ast(t) d\mu$. This implies due to \eqref{rep_reg_sub} that \begin{align*}
\langle \x^\ast (t) , u\rangle -\langle y^\ast,v\rangle\le d_{\gph\Phi_t}\big(\ox,\ooy(t)\big)(u,u)\; \text{ for all }\;(u,v)\in\X\times\Y, \end{align*} which ensures by definition of the convex
closure the estimate \begin{align*} \langle \x^\ast (t) , u\rangle  - \langle y^\ast , v\rangle \leq  \overline{\co} d_{\gph \Phi_t }\big(\ox,\ooy(t)\big) (u,v)\; \text{ for all }\;(u,v) \in \X
\times \Y. \end{align*} Fixing $(u,v) \in \X\times \Y$, find $\v$ satisfying \eqref{ineqsubd}, and so \begin{align*} \langle x^\ast, u\rangle  - \langle y^\ast , v\rangle &= \int_{{T}} \left(
\langle \x^\ast (t) , u\rangle  - \langle y^\ast , \v(t) \rangle\right) d\mu \\ & \le\int_{T} \overline{\co} d_{\gph \Phi_t } \big(\ox,\ooy(t)\big)\big(u,\v(t)\big) d\mu \\&\le d_{\gph
\Intfset{\Phi}} (\ox,\oy) (u,v). \end{align*} Invoking now \eqref{rep_reg_sub} yields $x^\ast\in \Hat{D}^\ast \Intfset{\Phi} \ox,\oy)(y^\ast)$, which verifies \eqref{cod-eq1} and thus completes
the proof. \end{proof}\vspace*{-0.07in}

The coderivative Leibniz rules obtained above in terms of the limiting coderivative and the coderivative characterizations of Lipschitzian properties of deterministic multifunctions discussed in
Section~\ref{sec3} allow us to establish efficient conditions for {\em Lipschitz stability} of expected-integral mappings. We present here the following result ensuring the Lipschitz-like
property of the multifunction $\Intfset{\Phi}$.\vspace*{-0.03in} \begin{proposition}[\bf Lipschitz-like property of expected-integral multifunctions]\label{eim-lip} In the setting of
Theorem~{\rm\ref{Theo_Basic_coder}}, assume in addition that \begin{equation}\label{eim-lip1} \int_T{D}^\ast\Phi_t\big(\ox,\ooy(t)\big)(0)d\mu=\big\{0\big\}\;\mbox{ for all
}\;\ooy\in\mathcal{S}_{\Phi}(\ox,\oy). \end{equation} Then the expected-integral multifunction $\Intfset{\Phi}$ from \eqref{def:set-valued:Exp} is Lipschitz-like around $(\ox,\oy)$.
\end{proposition}\vspace*{-0.05in} {\bf Proof}. We readily deduce this statement from the coderivative Leibnitz rule \eqref{lim-point} and the coderivative criterion \eqref{cod-cr} for the
Lipschitz-like property of deterministic set-valued mappings. $\h$\vspace*{0.07in}

We conclude this section with a direct consequence of the obtained coderivative Leibniz rules of the inclusion and equality types to derive the corresponding subdifferential Leibniz rule for
expected-integral functionals. The results in this vein can be found in different settings in \cite{chp19,mor-sag18,chp20,chp192}. \begin{proposition}\label{firstorderestimation} Let
$\varphi\colon T\times\R^n\to\Rex$ be a normal integrand. Take $\ox\in\X$ and assume that there exist a positive integrable function $\ell\colon T \to(0,\infty)$ and a number $\epsilon>0$ such
that \begin{align*}|\ph(t,x)-\ph(t,y)|\le\ell(t) \| x -y\|\;\text{ for all }\;x,y\in\mathbb{B}_\epsilon(\ox)\;\mbox{ and a.e. }\; t\in T. \end{align*} Then we have the limiting subdifferential
Leibniz rule \begin{equation}\label{subL} \partial\Intf{\ph}(\ox)\subset\int_T\partial\ph_t(\ox)d\mu,\end{equation} where $\ph_t(x):=\ph(t,x)$. Furthermore, \eqref{subL} holds as equality if
$\ph_t$ is lower regular at $\ox$ for a.e. $t\in T$. \end{proposition}\vspace*{-0.05in} {\bf Proof}. The inclusion result in \eqref{subL} follows directly from Corollary~\ref{cor_Basic_coder}. To
verify the equality therein, observe that $D^\ast\ph(\ox)(1)=\partial\ph(\ox)$ for real-valued l.s.c.\ functions (see \cite[Theorem~1.23]{m18}), and that we always have $\Hat
D^\ast\ph(\ox)(1)\supset\Hat \partial\ph(\ox)$. Invoking finally the assumed lower regularity verifies the equality in \eqref{subL} and thus completes the proof of the proposition.
$\h$\vspace*{0.05in}

\section{Lipschitz Stability and Coderivatives of Random Mappings with Composite Integrands}\label{sec5}\sce\vspace*{-0.1in}

This section is devoted to the study of expected-integral multifunctions $\Intfset{\Phi}$ from \eqref{def:set-valued:Exp} with $\Phi\colon T\times\X\tto\Y$ generated by {\em composite} set-valued
normal integrands of the type \begin{equation}\label{compos} \Phi_t(x)=F\big(t,g_t(x)\big)\;\text{ for all }\;x\in U\;\text{ and a.e. }\;t\in T \end{equation} near some point
$\ox\in\dom\Intfset{\Phi}$, where $F\colon T\times\Z\tto\Y$ and $g\colon T\times U\to\Z$. First we reveal general conditions on $F$ and $g$ ensuring that $\Phi$ is the {\em integrably locally
Lipschitzian} in the sense of \eqref{eq_definition_Int_loc} around the reference point. Then we introduce a new notion of {\em integrable amenable multifunctions}, which expands to (both random
and deterministic) set-valued mappings the fundamental notion of amenable extended-real-valued deterministic functions that was comprehensively studied in \cite{rw}. The Leibniz-type rules
obtained for such compositions give us {\em pointwise} upper estimates of both regular and limiting coderivatives of \eqref{def:set-valued:Exp} via compositions of limiting coderivatives under
the integral sign. Finally, we apply the obtained results to evaluate coderivatives of expected-integral multifunctions generated by feasible set mappings in constrained {\em stochastic
programming}.\vspace*{0.05in}

Let us start with establishing {\em Lipschitz stability} of expected-integral multifunctions with composite integrals \eqref{compos}. The following lemma for deterministic set-valued mappings with convex
graphs and its very simple proof are of their own interest. This result can be treated as a {\em quantitative} counterpart (without the closed-graph assumption) of \cite[Theorems~5.9 and 5.12]{m93}.
Recall \cite{r} that a set-valued mapping $G\colon\X\tto\Y$ is {\em sub-Lipschitzian} around $\ox\in\dom G$ if for every compact set $V\subset\Y$ there exist a neighborhood $U$ of $\ox$ and a number
$\ell\ge 0$ such that we have the inclusion
\begin{equation}\label{sub-lip}
G(x)\cap V\subset G(x')+\ell\|x-x'\|\B\;\mbox{ for all }\;x,x'\in U.
\end{equation}
The coderivative characterization from \cite[Theorem~5.9(b)]{m93} of the sub-Lipschitzian property \eqref{sub-lip} for (locally) closed-graph multifunctions reads as follows: $G$ is sub-Lipschitzian
around $\ox$ if and only if for any compact set $V\subset\Y$ there exist $\eta>0$ and $\ell\ge 0$ such that
\begin{equation}\label{sub-lip1}
\sup\big\{\|x^*\|\;\big|\;x^*\in D^*G(x,y)(y^*)\big\}\le\ell\|y^*\|\;\mbox{ whenever }\;y^*\in\Y
\end{equation}
for all $x\in\B_\eta(\ox)$ and $y\in G(x)\cap V$. Here is our quantitative version for convex-graph multifunctions.\vspace*{-0.05in}

\begin{lemma}[coderivative estimate for convex-graph multifunctions]\label{lemma:locLip} Let $G\colon\X\tto\Y$ be a convex-graph multifunction with $\ox\in\dom G$ for which there exist positive numbers
$M$ and $\eta$ such that
\begin{equation}\label{sub-ass}
{\rm dist}\big(0;G(x)\big)\le M\;\text{ whenever }\;x\in\mathbb{B}_{2\eta}(\bar{x}).
\end{equation}
Then for all $x\in\mathbb{B}_{\eta}(\bar{x})$, all $y\in G(x)$, and all $y^*\in\Y$ we have the coderivative estimate
\begin{equation}\label{est+}
\sup\big\{\|x^\ast\|\;\big|\;x^\ast\in D^\ast G(x,y)(y^\ast)\big\}\le\left(\frac{M+\|y\|}{\eta}\right)\|y^\ast\|.
\end{equation}
If in addition the graph of $G$ is locally closed, then this multifunction is sub-Lipschitzian around $\ox$.
\end{lemma}\vspace*{-0.05in}
{\bf Proof}. Pick any $x\in\mathbb{B}_{\eta}(\bar{x})$, $y\in G(x)$, and $x^\ast\in D^\ast G(x,y)(y^\ast)$ with $y^\ast\in\Y$. Since the graph of $G$ is convex and \eqref{sub-ass} holds, we get the
inequalities (see, e.g., \cite[Proposition~1.7]{m18})
\begin{equation*}
\langle x^\ast,x+h-x\rangle\le\langle y^\ast,y_h-y\rangle\le\|y^\ast\|(M+\|y\|)\;\mbox{ for all }\;h\in\mathbb{B}_{\eta}(0),
\end{equation*}
where $y_h\in G(x+h)$ is such that $\| y_h\|={\rm dist}(0;G(x+h))$. This readily yields \eqref{est+}. Finally, the sub-Lipschitzian property of $G$ around $\ox$ follows from  \eqref{sub-lip1} and
\eqref{est+} with $\ell:=(M+\|y\|)/\eta$. $\h$\vspace*{0.05in}

The second lemma is more technical while being used below in the proofs of both theorems in this section. To proceed, we need the following assumption: for all $h\in\B_\eta(\ox)$ the function
\begin{equation}\label{a1}
t\mapsto{\rm dist}\big(0;F_t(g_t(\ox)+h)\big)\;\mbox{ is integrable on }\;T.
\end{equation}
\begin{lemma}[\bf perturbed distance estimate]\label{dist-lem} Let $F\colon T\times\Z\tto\Y$ be a convex normal integrand satisfying condition \eqref{a1}. Then there exist $\kappa\in\Leb^1(T,\R_+)$ and
$\eta_1\in(0,\eta)$ such that
\begin{equation}\label{a1+}
\dist\big(0;F_t(g_t(\bar x)+h)\big)\le\kappa(t)\;\text{ for all }\;h\in\eta_1\mathbb{B}\;\text{ and a.e. }\;t\in T.
\end{equation}
\end{lemma}\vspace*{-0.03in}
{\bf Proof}. Take finitely many vectors $e_i\in\eta\mathbb{B}$ as $i\in I$ such that $0\in\inter\Delta$, where $\Delta$ is the simplex generated by vectors $\{e_i\}_{i\in I}$, i.e., given by
\begin{equation*}
\Delta:=\bigg\{\sum_{i\in I}\lambda_i e_i\;\bigg|\;\lambda_i\ge 0,\;\sum_{i\in I}\lambda_i=1\bigg\}.
\end{equation*}
Define the function $\kappa(t):=\max\{{\rm dist}(0;F_t(g_t(\bar x)+e_i))\;|\;i\in I\}$, which is integrable on $T$ by assumption \eqref{a1}, and then pick any $h$ in the form $h=\sum_{i\in I}
\lambda_i e_i$. Recalling that the graph of $F_t$ is convex for a.e.\ $t\in T$ ensures the fulfillment of the estimates
\begin{equation*}
{\rm dist}\big(0;F_t(g_t(\bar x)+h)\big)\le\sum_{i\in I}\lambda_i{\rm dist}\big(0;F_t(g_t(\bar x)+e_i)\big)\le\kappa(t)\;\mbox{ for a.e. }\;t\in T.
\end{equation*}
Choosing finally $\eta_1>0$ with $\eta_1\mathbb{B}\subset\Delta$ verifies \eqref{a1+} and thus completes the proof of the lemma. $\h$\vspace*{0.05in}

Having Lemmas~\ref{lemma:locLip} and \ref{dist-lem} in hand, we now establish the integrable local Lipschitzian property of convex composite normal integrands \eqref{compos} under some additional
assumptions.\vspace*{-0.01in} \begin{theorem}[\bf integrable local Lipschitzian property of composite integrands]\label{lip-comp} Let $\Phi\colon T\times\X\tto\Y$ be a normal integrand
represented in the composite form \eqref{compos} around $\ox\in\dom\Intfset{\Phi}$ on a complete finite measure space $(T,{\cal A},\mu)$, where $F\colon T\times\Z\tto\Y$ is a set-valued convex
normal integrand that is integrably bounded as in \eqref{convex_cond_set}, and where $g_t(x)$ is continuously differentiable around $\ox$ with the uniformly bounded gradients $\nabla g_t(x)$ for
all $x\in\B_\eta(\ox)\subset U$ and for a.e.\ $t\in T$ under the fulfillment of \eqref{a1}. Then the multifunction $\Phi$ is integrably locally Lipschitzian around $\ox$.
\end{theorem}\vspace*{-0.05in} {\bf Proof}. Select a set $\Hat T\subset T$ of full measure, and then take $\eta>0$ and $\kappa\in\Leb^1(T,\R_+)$ such that $F$ satisfies \eqref{a1+} for all
$t\in\Hat{T}$, and that $\Phi$ satisfies the integrable bounded condition from \eqref{convex_cond_set} over $\mathbb{B}_{\eta}(\ox)$ on $\Hat T$ with this function $\kappa$. Now choose
$\epsilon\in(0,\eta)$ ensuring the inequalities \begin{equation*} \sup\big\{\|\nabla g_t(x)\|\;\big|\;(t,x)\in\Hat{T}\times\mathbb{B}_{\epsilon}(\ox)\big\}\le\ell\;\mbox{ and} \end{equation*}
\begin{equation*} \|g_t(x)-g_t(u)\|\le\ell\|x-u\|\;\mbox{ whenever}\;x,u\in\mathbb{B}_{\epsilon}(\ox),\;t\in\Hat{T} \end{equation*} for some $\ell>0$. Suppose without loss of generality that
$\epsilon\ell<\eta$. Now taking $t\in\Hat{T}$, $x\in\mathbb{B}_{\epsilon/2}(\ox)$, $y\in\Phi_t(x)$, and $x^\ast\in D^\ast\Phi_t(x,y)$, we observe that $g_t(x)\in\mathbb{B}_{\eta/2}(g_t(\ox))$.
Applying Lemma~\ref{lemma:locLip} to the mapping $G:=F_t$ at the point $g(\ox)$ gives us the coderivatives estimate \eqref{est+}, which ensures by the coderivative criterion \eqref{cod-cr} that
$F_t$ is Lipschitz-like around $(g_t(x),y)$. Then applying the coderivative chain rule for deterministic multifunctions from \cite[Theorem~3.11(iii)]{m18} to the composition in \eqref{compos}
yields the representation \begin{equation*} x^\ast=\nabla g_t(x)^\ast\circ z^\ast\;\mbox{ with some }\;z^\ast\in D^\ast F_t\big(g(x),y\big)(y^\ast)\;\mbox{ and }\;y^*\in\Y. \end{equation*}
Employing again by the coderivative estimate \eqref{est+} leads us to the inequality \begin{equation*} \|z^\ast\|\le\frac{4\kappa(t)}{\eta}\|y^\ast\|, \end{equation*} where we use that
$\|y(t)\|\le\kappa(t)$ due to the choice of $y\in\Phi_t(x)$. This tells us that \begin{equation*} \| x^\ast\|=\|\nabla g_t(x)^\ast\circ z^\ast\|\le\|\nabla
g_t(x)^\ast)\|\cdot\|z^\ast\|\le\frac{4\ell\kappa(t)}{\eta}\|y^\ast\|. \end{equation*} Since the obtained estimate holds for all $t\in\Hat{T}$, $x\in \mathbb{B}_{\epsilon/2}(\ox)$, and
$y\in\Phi_t(x)$, we arrive at the coderivative condition \eqref{Int_Lips_like_inq_nonatomic}, which ensures by Theorem~\ref{Prop47} that $\Phi$ is integrably locally Lipschitzian around $\ox$.
$\h$\vspace*{0.05in}

Next we introduce a new property of random multifunctions that plays a crucial role in deriving composite coderivative Leibniz rules in what follows.\vspace*{-0.05in} \begin{definition}[\bf
integrably amenable multifunctions]\label{amen} A set-valued normal integrand $\Phi\colon T\times\X\to\Z$ defined on a complete finite measure space $(T,\mathcal{A},\mu)$ is {\em integrably
amenable} at $(\ox,\oy)\in\gph\Intfset{\Phi}$ if there exist a neighborhood $U$ of $\ox$, a set-valued convex normal integrand $F\colon T\times\Z\tto\Y$, and a function $g\colon T\times U\to\Z$,
which is measurable with respect to $t$ on $T$ and continuously differentiable with respect to $x$ around $\ox$, such that the composite representation \eqref{compos} holds, and that the
following {\em qualification condition} is satisfied for a.e.\ $t\in T$: \begin{equation}\label{qc} D^\ast F_t\big(g_t(\ox),\ooy(t)\big)(0)\cap\operatorname{Ker}\nabla g_t(\ox)^*=\{0\}\;\mbox{
whenever }\;\ooy\in{\cal S}_\Phi(\ox,\oy). \end{equation} \end{definition}\vspace*{-0.05in}

It follows from Definition~\ref{int-lipl}(ii) that the qualification condition \eqref{qc} fulfills automatically if the outer mapping $F$ in \eqref{compos} is {\em integrably quasi-Lipschitzian}
around $(\ox,\ooy)$ for each $\ooy\in{\cal S}_F(\ox,\oy)$. Recall that Theorem~\ref{Prop4700} ensures the fulfillment of \eqref{qc} when $F$ is integrably Lipschitz-like around $(\ox,\ooy)$ for
each $\ooy\in{\cal S}(\ox,\oy)$, provided that $(T,\mu,\mathcal{A})$ is a purely atomic measure space consisting of countable disjoint family of atoms. In the alternative setting of nonatomic
spaces, we deduce from Theorem~\ref{nonatomic:charact} that the integrable quasi-Lipschitzian property of $F$ to ensure \eqref{qc} can be equivalently replaced but its locally Lipschitzian
counterpart. On the other hand, observe from \eqref{qc} that this condition is satisfied, independently of the Lipschitzian properties of $F$, if the Jacobian $\nabla g_t(\ox)$ is of {\em full
rank} for a.e.\ $t\in T$.\vspace*{0.05in}

Now we are ready to derive the pointwise coderivative Leibniz rules for both regular and limiting coderivatives of expected-integral multifunctions with amenable set-valued
integrands.\vspace*{-0.05in} \begin{theorem}[\bf coderivative Leibniz rules for integral multifunctions with amenable integrands]\label{theoremamenablelike} Let $\Phi\colon T\times\X\tto\Y$ be an
amenable multifunction at $(\ox,\oy)\in\gph\Intfset{\Phi}$, which is assumed to be integrably bounded as in \eqref{convex_cond_set} with the uniformly bounded gradients $\nabla g_t(x)$ while
satisfying the integrability condition \eqref{a1} for the outer mapping $F$ in \eqref{compos}. Then for every $\ooy\in \mathcal{S}_\Phi(\ox,\oy)$ and every $y^\ast\in\Y$ the regular coderivative
upper estimate \begin{equation}\label{leib-comp} \Hat{D}^\ast\Intfset{\Phi}(\ox,\oy)(y^\ast)\subset\int_T\nabla g_t(\ox)^\ast\circ D^\ast F_t\big(g_t(\ox),\ooy(t)\big)(y^\ast)d\mu \end{equation}
holds. If in addition the multifunction $\mathcal{S}_\Phi$ is inner semicompact at $(\ox,\oy)$, then we have the following upper estimate of the limiting coderivative of $\Intfset{\Phi}$ at
$(\ox,\oy)$: \begin{equation}\label{secondinclusionformula} D^\ast\Intfset{\Phi}(\ox,\oy)(y^\ast)\subset\bigcup\limits_{\ooy\in\mathcal{S}_{\Phi}(\ox,\oy )}\int_T\nabla g_t(\ox)^\ast\circ{D}^\ast
F_t\big(g_t(\ox),\ooy(t)\big)(y^\ast)d\mu. \end{equation} \end{theorem}\vspace*{0.05in} {\bf Proof}. It follows from Theorem \ref{lip-comp} that $\Phi$ is integrably locally Lipschitzian around
$\ox$. Then we apply Theorem~\ref{Theo_Basic_coder} to observe that the coderivative Leibniz rule \eqref{leib-comp} holds with the integral of the set-valued mapping $D^\ast\Phi_t(\ox,\ooy(t))$
on the the right-hand side. Applying now the coderivative chain rule for deterministic multifunctions from \cite[Theorem~3.11(iii)]{m18} under the amenability assumption of Definition~\ref{amen},
we arrive at the claimed assertion \eqref{leib-comp}. Finally, the limiting coderivative inclusion \eqref{secondinclusionformula} is verified similarly to the above proof by involving the
corresponding arguments of Theorem~\ref{Theo_Basic_coder}. $\h$\vspace*{0.07in}

The concluding part of this section provides an efficient specification of the main Theorem~\ref{theoremamenablelike} for the case of {\em random inequality constraint systems} described by smooth
functions. In the deterministic framework, such constraint systems appear in nonlinear programming, while in the random setting under consideration they address parametric sets of feasible solutions in
problems of stochastic programming. Let us formalize this as follows. Given $\varphi^i\colon T\times\Z\times\Y\to\R$  a finite family of convex normal integrands with $i\in I$, assume that $\ph^i_t(z,y)$
are continuously differentiable with respect to $(z,y)$ around the reference points for a.e.\ $t\in T$ and then define the set-valued convex normal integrand $F\colon T\times\Z\tto\Y$ by
\begin{equation}\label{constraint:system}
F(t,z):=\big\{y\in\Y\;\big|\;\varphi^i_t(z,y)\le 0\;\text{ for all }\;i\in I\big\}.
\end{equation}

For brevity, we skip here establishing specifications of Theorem~\ref{lip-comp} on Lipschitz stability for the case of integrands \eqref{constraint:system}, while concentrating  on deriving
coderivative Leibniz rule for such systems. To proceed, let us introduce two {\em constraint qualification conditions}. The first qualification condition can be treated as a random uniform
version of the classical Slater constraint qualification in deterministic convex programming. The second one requires the solution triviality for a certain adjoint system generated by the
derivatives of $\ph^i_t$ and $g_t$ at
he reference points.\vspace*{-0.05in} \begin{definition}[\bf integrable constraint qualifications]\label{iqc} Consider the random constraint system
\eqref{constraint:system}.\\[0.5ex] {\bf(i)} Given a measurable mapping $\z\colon T\to\Z$, we say that the constraint system \eqref{constraint:system} satisfies the {\em integrable Slater
constraint qualification} at $\z$ if there is a number $\eta>0$ such that for all $h\in\eta\mathbb{B}$ there exists $\y\in\Leb^1(T,\Y)$ ensuring the strict inequality
\begin{equation}\label{integrable:slater} \varphi^i_t\big(\z(t)+h,\y(t)\big)<0\;\text{ whenever }\;i\in I\;\text{ for a.e. }\;t\in T. \end{equation} {\bf(ii)} Fix $x\in\X$, $t\in T$, and $y\in
F_t(g_t(x))$, and then define the {\em adjoint system} at $(t,x,y)$ by \begin{equation}\label{GE} (z^\ast,0)=\hspace{-0.3cm}\sum_{i\in
I_t(x,y)}\hspace{-0.3cm}\lambda_i\nabla\varphi^i_t\big(g_t(x),y\big),\;\lambda_i\ge 0,\;\nabla g_t(x)^\ast(z^\ast)=0, \end{equation} where $I_t(x,y):=\{i\in I\;|\;\varphi^i_t(g(x),y)=0\}$. We say
that the {\em integral triviality qualification condition} (ITQC) holds at $\ox$ if there exists $\eta>0$ such that for all $x\in \mathbb{B}_\eta(\ox)$, a.e.\ $t\in T$, and all $y\in F_t(g_t(x))$
the adjoint system \eqref{GE} admits only the trivial solution $z^\ast=0$. \end{definition}\vspace*{-0.05in}

The following consequence of Theorem~\ref{theoremamenablelike} establishes pointwise coderivative Leibniz rules for expected-integral multifunctions $\Intfset{\Phi}$ with composite integrands
\eqref{compos}, where $F$ is represented in the constraint form \eqref{constraint:system}. The obtained coderivative estimates are given entirely in terms of the {\em initial constraint data}.

\begin{corollary}[\bf coderivative Leibniz rules over random constraint systems]\label{rand-const} Let $\Phi\colon T\times\X\tto\Y$ be an integrably bounded normal integrand on a complete finite
measure space $(T,{\cal A},\mu)$ given in the composite form \eqref{compos} around $\ox$ with some fixed pair $(\ox,\oy)\in\gph\Intfset{\Phi}$, where $F\colon T\times\Z\tto\Y$ is the constraint
convex-graph multifunction defined in \eqref{constraint:system}, and where $g\colon T\times U\to\Z$ is measurable in $t$ and continuously differentiable in $x$ while satisfying
\begin{equation}\label{g-bound} \sup\big\{\|\nabla g_t(x)\|\;\big|\;(t,x)\in U\times T\big\}<\infty. \end{equation} Assume further that the ITQC holds at $\ox$, and that the integrable Slater
constraint qualification is satisfied at $\z:=g_t(\ox)$. Then for every $\ooy\in\mathcal{S}_\Phi(\ox,\oy)$ and every $y^\ast\in\Y$ we have the inclusion \begin{equation}\label{reg-const}
\Hat{D}^\ast\Intfset{\Phi}(\ox,\oy)(y^\ast)\subset\int_T\nabla g_t(\ox)^\ast\circ D^\ast F_t\big(g_t(\ox),\ooy(t)\big)(y^\ast)d\mu, \end{equation} where the limiting coderivative $D^\ast
F_t(g_t(\ox),\ooy(t))$ is computed by \begin{align}\label{formulacodF} D^\ast F_t\big(g_t(\ox),\ooy(t)\big)(y^\ast)=\left\{z^\ast\in\Z\;\Bigg|\;\begin{array}{c}
(z^\ast,-y^\ast)=\hspace{-0.4cm}\displaystyle\sum\limits_{i\in I_t(\ox,\oy(t))}\lambda_i\nabla\varphi^i_t\big(g_t(\ox),\oy(t)\big)\\\vspace{-0.3cm}\\ \text{ for some }\;\lambda_i\ge 0\;\text{
with }\;i\in I_t\big(\ox,\oy(t)\big) \end{array}\right\}. \end{align} If in addition the multifunction $\mathcal{S}_\Phi$ is inner semicompact at $(\ox,\oy)$, then we have
\begin{equation}\label{lim-const} D^\ast\Intfset{\Phi}(\ox,\oy)(y^\ast)\subset\bigcup\limits_{\ooy\in\mathcal{S}_{\Phi}(\ox,\oy)}\int_T\nabla g_t(\ox)^\ast\circ{D}^\ast
F_t\big(\ox,\ooy(t)\big)(y^\ast)d\mu, \end{equation} where $D^\ast F_t\big(\ox,\ooy(t))$ is computed in \eqref{formulacodF}. \end{corollary} {\bf Proof}. First we check that all the assumptions
of Theorem~\ref{theoremamenablelike} are fulfilled under the assumptions made in this corollary. Indeed, the required (local) uniform boundedness of the gradients $\nabla g_t(x)$ is formalized in
\eqref{g-bound}, while the imposed integrable Slater constraint qualification \eqref{integrable:slater} readily ensures the existence of $\eta_1>0$ for which condition \eqref{a1} holds whenever
$h\in\B_{\eta_1}(\ox)$. Further, choose $\eta_2>0$ such that for all $x\in\mathbb{B}_{\eta_2}(\ox)$, a.e.\ $t\in T$, and all $y\in F_t(g_t(x))$ we have that the adjoint system \eqref{GE} admits
only the trivial solution $z^\ast=0$. Denoting $\eta:=\min\{\eta_1/(\ell+1),\eta_2\}$ and fixing the triple $(x,y,t)$ from the above ensure that $g_t(x)-g_t(\ox)\in\eta_1 \mathbb{B}$. Employing
again the integrable Slater condition together with \cite[Corollary~4.35]{m06} (by taking into account that the Slater condition yields in this setting the Mangasarian-Fromovitz constraint
qualification assumed in \cite[Corollary~4.35(b)]{m06}), we arrive at the coderivative calculation \begin{align*}
 D^\ast F_t\big(g_t(x),y\big)(v^\ast)=\left\{z^\ast\in\Z\;\Bigg|\;\begin{array}{c}
(z^\ast,-v^\ast)=\hspace{-0.4cm}\displaystyle\sum\limits_{i \in I_t(x,y)}\lambda_i\nabla\varphi^i_t\big(g_t(x),y\big)\\\vspace{-0.3cm}\\ \text{ for some }\;\lambda_i\ge 0\;\text{ with }\;i\in
I_t(x,y) \end{array}\right\} \end{align*} whenever $v^*\in\Y$; this clearly implies \eqref{formulacodF}. It now follows from the ITQC that the qualification condition \eqref{qc} is satisfied, and
thus $\Phi$ is an integrably amenable multifunction at $(\ox,\oy)$. Applying the coderivative Leibniz rules from Theorem~\ref{theoremamenablelike} yields both inclusions in \eqref{reg-const} and
\eqref{lim-const}.$\h$\vspace*{0.05in}

To conclude this section, observe that the obtained basic coderivative Leibniz rules in \eqref{secondinclusionformula} and \eqref{lim-const} readily imply, similarly to the proof of
Corollary~\ref{eim-lip1} by using the coderivative criterion \eqref{cod-cr}, the {\em Lipschitz-like property} \eqref{aub} of expected-integral multifunctions $\Intfset{\Phi}$ with composite integrands
\eqref{compos}, as well as their specifications for random constraint systems \eqref{constraint:system}.\vspace*{-0.2in}

\section{Second-Order Subdifferentials of Expected-Integral Functionals}\label{sec6}\sce\vspace*{-0.1in}

In this section we   study of expected  functionals $\Intf{\ph}(x)$ of type \eqref{eif:expected} generated by extended-real-valued normal integrands $\ph\colon T\times\X\to\Rex$ on complete finite
measure spaces $(T,{\cal A},\mu)$. We refer the reader to \cite{chp19,chp192,chp20,mor-sag18} and the bibliographies therein for more results concerning first-order subdifferential calculus rules of
convex and nonconvex expected-integral functionals of type \eqref{eif:expected}. Our main focus here is on sequential and pointwise {\em second-order Leibniz rules} for $\Intf{\ph}(x)$ derived in terms
of both basic and combined second-order subdifferentials that are defined in \eqref{2nd} and \eqref{2nd1}, respectively. \vspace*{0.03in}

To simplify our analysis, we postulate that the {\em first-order} subdifferential Leibniz rule holds as an {\em equality}, which means that for the reference point $\ox\in\dom\Intf{\ph}$ there exists
$\rho>0$ such that
\begin{equation}\label{cond_reg_sub}
\partial\Intf{\ph}(x)=\int_T\partial\ph_t(x)d\mu\;\text{ whenever }\;x\in\mathbb{B}_\rho(\ox).
\end{equation}

The following proposition reveals some important cases where condition \eqref{cond_reg_sub} is satisfied.\vspace*{-0.05in} \begin{proposition}[\bf first-order subdifferential Leibniz rule as
equality]\label{suf_condition_to_cond_reg_sub} Let $\ph$ be a normal integrand on a complete finite measure space $(T,{\cal A},\mu)$, and let $\ox\in\dom\Intf{\ph}$. Then we have
\eqref{cond_reg_sub} provided that either the assumptions in {\rm(i)} or those in {\rm(ii)} hold:

\item[\bf(i)] $\ph_t$ is convex for almost all $t\in T$, and $\ox$ is an interior point of $\dom\Intf{f}$. \item[\bf(ii)] There exist $\Hat T\in\mathcal{A}$ with $\mu(T\backslash\Hat T)=0$,
$\kappa\in\Leb^1(T,\R_+)$ and a neighborhood $U$ of $\ox$ such that \begin{equation}\label{integrable_lipschitzlike} |\ph(t,u)-\ph(t,v)|\le\kappa(t)\|u-v\|\;\text{ for all }\;u,v\in U\;\text{ and
}\;t\in\Hat T, \end{equation} and for all $x\in U$ the function $\ph_t $ is lower regular at $x$ for a.e.\ $t\in T$.\\[1ex] Furthermore, in both these cases there exists $\rho>0$ such that
\begin{equation}\label{cond_reg_sub:extra} \Hat\partial\Intf{\ph}(x)=\partial\Intf{\ph}(x)=\int_T\Hat\partial\ph_t(x)d\mu=\int_T \partial\ph_t(x)d\mu\;\text{ for all }\;x\in\mathbb{B}_\rho(\ox).
\end{equation} \end{proposition}\vspace*{-0.05in} {\bf Proof}. The claimed result in case (i) with the formulas in \eqref{cond_reg_sub:extra} can be found, e.g., in \cite{chp19,chp192}. The
verification of \eqref{cond_reg_sub:extra}, and hence of \eqref{cond_reg_sub}, in case (ii) follows from Proposition~\ref{firstorderestimation}. $\h$\vspace*{0.05in}

To proceed further, we assume from now on that
\begin{equation}\label{cond_reg_sub0}
\begin{array}{ll}
\ph_t\;\text{ is lower regular at }\;x\in\mathbb{B}_\rho(\ox)\;\text{ and all }\;t\in T_{na},\\
\partial\ph_t(x)\subset\kappa(t)\mathbb{B}\;\text{ for all }\;x\in\mathbb{B}_\rho(\ox)\;\text{ and a.e. }\;t\in T,
\end{array}
\end{equation}
where a constant $\rho>0$ and an integrable function $\kappa\in\Leb^1(T,\R_+)$ are fixed in what follows. Similarly to \eqref{mapping:SPHI} consider the multifunction
$\mathcal{S}_\ph\colon\X\tto\Leb^1(T,\X)$ defined by
\begin{equation}\label{Sph}
\mathcal{S}_{\ph}(x,y):=\Big\{\y\in\Leb^1(T,\Y)\;\Big|\;\int_T\y(t)d\mu=y\;\text{ and }\;\y(t)\in\partial\ph_t(x)\;\text{ for a.e. }\;t\in T\Big\}.
\end{equation}

Let us start the derivation of second-order Leibniz rules with {\em sequential} results involving the {\em combined} second-order subdifferential mappings \eqref{2nd1}. These results are induced by
the corresponding sequential Leibniz rules for coderivatives obtained in Section~\ref{sec4}. Here are the two statements in this direction. For brevity we proof only the second one by taking into
account that the proof of the first proposition is pretty similar with the usage of Proposition~\ref{main_fuzzy_cod2} instead of Proposition~\ref{main_fuzzy_cod}.\vspace*{-0.05in}
\begin{proposition}[\bf sequential second-order subdifferential Leibniz rule, I]\label{2leib-1} Let $\ph\colon T\times\X\to\Rex$ be a normal integrand with $\ox\in\dom\Intf{\ph}$ and $\oy\in\partial
\Intf{\ph}(\ox)$. Suppose that $\ph$ satisfies conditions \eqref{cond_reg_sub} and \eqref{cond_reg_sub0} around $\ox$ and pick $\ooy\in\mathcal{S}_{\ph}(\ox,\oy)$ such that $\oy=\int_T\ooy(t)d\mu$ and
$\ooy(t)\in\partial\ph_t(\ox)$ for a.e.\ $t\in T$. Then for every $p,q\in(1,\infty)$ with $1/p+1/q=1$ there exist sequences $\{x_k\}\subset\X$, $\{\x_k\}\subset\Leb^p({T},\X)$, $\{\x_k^*\}
\subset{\Leb}^q({T},\X)$, $\{\y_k\}\subset\Leb^1(T,\X)$, and $\{\y_k^\ast\}\subset\Leb^\infty(T,\X)$ for which the following hold:\vspace*{-0.05in}
\begin{enumerate}[label=\alph*),ref=\alph*)]
\item[\bf(i)] $\x_k^*(t)\in \breve{\partial}^2\ph_t\big(\x_k(t),\y_k(t)\big)\big(\y_k^\ast(t)\big)$ for a.e.\ $t\in T$ and all $k\in\N$.
\item[\bf(ii)] $\|\ox-x_k\|\to 0$, $\|\ox -\x_k \|_p\to 0$, and $\|\ooy-\y_k\|_1\to 0$ as $k\to\infty$.
\item[\bf(iii)] $\|\y_k^\ast-\ooy^\ast\|_\infty\to 0$, $\|\x_k^*\|_q\|\x_k-x_k\|_p\to 0$, and $\disp\int_T x_k^*(t)d\mu\to\ox^\ast$ as $k\to\infty$.
\end{enumerate}
\end{proposition}\vspace*{-0.15in}
\begin{proposition}[\bf sequential second-order subdifferential Leibniz rule, II]\label{main_fuzzy_sec_ord} In the general setting of Proposition~{\rm\ref{2leib-1}} $($without the specification of $p$
and $q$$)$, take any $\ox^\ast\in\breve{\partial}^2\Intfset{\ph}(\ox,\oy^\ast)(\oy^\ast)$. Then there exist sequences $\{x_k\}\subset\X$, $\{\x_k\}\subset\Leb^\infty({T},\X)$, $\{\x_k^\ast\}
\subset{\Leb}^1({T},\X)$, $\{\y_k\}\subset\Leb^1(T,\X)$, and $\{\y_k^\ast\}\subset\Leb^\infty(T,\X)$ satisfying the following conditions:
\begin{enumerate}[label=\alph*),ref=\alph*)]
\item[\bf(i)] $\x_k^*(t)\in\breve{\partial}^2\ph_t\big(\x_k(t),\y_k(t)\big)\big(\y_k^\ast(t)\big)$ for a.e.\ $t\in T$ and all $k\in\N$.
\item[\bf(ii)] $\|\ox-x_k\|\to 0$, $\|\ox-\x_k\|_\infty\to 0$, $\|\ooy-\y_k\|_1\to 0$, and $\|\y_k^\ast-\oy^\ast\|_\infty\to 0$ as $k\to\infty$.
\item[\bf(iii)]$\disp\int_T\|\x_k^*(t)\|\cdot\|\x_k(t)-x_k\|d\mu\to 0$ and $\disp\int_T\x_k^*(t)d\mu\to\ox^\ast$ as $k\to\infty$.
\end{enumerate}
\end{proposition}\vspace*{-0.05in}
{\bf Proof}. Define the set-valued normal integrand $\Phi\colon T\times\X\tto\X$ by $\Phi_t(x):=\partial\ph_t(x)$. It follows from \eqref{cond_reg_sub0} that $\Phi$ satisfies the conditions in
\eqref{convex_cond_set}. Furthermore, by \eqref{cond_reg_sub} we have that
\begin{equation*}
\partial\Intf{\ph}(x)=\Intfset{\Phi}(x)\;\text{ for all }\;x\in\mathbb{B}_\rho(\ox).
\end{equation*}
Applying now Proposition~\ref{main_fuzzy_cod} gives us sequences $\{x_k\}\subset\X$, $\{\x_k\}\subset\Leb^\infty({T},\X)$, $\{\x_k^\ast\}\subset{\Leb}^1({T},\X)$, $\{\y_k\}\subset\Leb^1(T,\Y)$,
and $\{\y_k^\ast\}\subset\Leb^\infty(T,\Y)$ satisfying the inclusions $\x_k^*(t)\in\Hat{D}^\ast\Phi_t(\x_k(t),\y_k(t))(\y_k^\ast(t))$ for a.e.\ $t\in T$ together with the convergence $\|\ox-x_k\|\to 0$,
$\|\ox-\x_k\|_\infty\to 0$, $\|\y_k^\ast-\oy^\ast\|_\infty \to 0$,
\begin{equation*}
\disp\int_T\|\ooy(t)-\y_k(t)\|d\mu\to 0,\;\int_T\|\x_k^*(t)\|\cdot\|\x_k(t)-x_k\|d\mu\to 0,\;\mbox{ and }\;\int_T\x_k^*(t)d\mu\to\ox^\ast\;\mbox{ as }\;k\to\infty.
\end{equation*}
Remembering that $\Hat{D}^\ast\Phi_t(\x_k(t),\y_k(t))(\y_k^\ast(t))=\breve{\partial}^2\ph_t(\x_k(t),\y_k(t))(\y_k^\ast(t))$, we conclude the proof.$\h$\vspace*{0.05in}

Our major goal in this section is to obtain {\em pointwise} second-order subdifferential Leibniz rules. To proceed, let us first specify the integrable quasi-Lipschitzian property of random
multifunctions from Definition~\ref{int-lipl}(ii) for the case of {\em basic} second-order subdifferential mappings \eqref{2nd}. Given a normal integrand $\ph\colon T\times\X\to\Y$, pick
$\ox\in\dom\Intf{\ph}$, $\oy\in\partial\Intf{\ph}(\ox)$, and $\ooy\in \Leb^1(T,\Y)$ with $\oy=\int_T\ooy (t)d\mu$ and $\ooy(t)\in\partial\ph_t(\ox)$ for a.e.\ $t\in T$. We say that $\ph$ enjoys
the {\em second-order integrable quasi-Lipschitzian property} around $(\ox,\ooy)$ if there exist $\eta>0$ and $\ell\in\Leb^1(T,\R_+)$ such that \begin{equation}\label{second_order_
Int_Lips_like_inq} \sup\big\{\|x^\ast\|\;\big|\;x^\ast\in\partial^2\ph_t\big(\x(t),\y(t)\big)\big(\y^\ast(t)\big)\big\}\le\ell(t)\|\y^\ast(t)\|\;\text{ for a.e. }\;t\in T \end{equation} whenever
$\x\in\mathbb{B}_\eta(\ox)$, $\y\in\mathbb{B}_\eta(\ooy)\cap\partial\ph(x)$, and $\y^\ast\in\Leb^\infty(T,\Y)$ with \begin{equation*}
\mathbb{B}_\eta(\ooy)\cap\partial\ph(x):=\big\{\y\in\Leb^1(T,\Y)\;\big|\;\y\in\mathbb{B}_\eta(\ooy)\;\text{ and }\;\y(t)\in\partial\ph_t(x)\;\text{ for a.e. }\;t\in T\big\}. \end{equation*}

The first theorem provides a pointwise Leibniz-type estimate of the {\em combined} second-order subdifferential of $\ph$ in terms of the basic second-order subdifferential of the integrand.
\vspace*{-0.05in}
\begin{theorem}[\bf pointwise Leibniz-type estimate of the combined second-order subdifferential]\label{2leib-comb} Let $\ph\colon T\times\X\to\Rex$ be a normal integrand defined on a complete finite
measure space $(T,{\cal A},\mu)$ with $\ox\in\dom\Intf{\ph}$ and $\oy\in\partial\Intf{\ph}(\ox)$. Suppose that $\ph$ satisfies conditions \eqref{cond_reg_sub} and \eqref{cond_reg_sub0} and enjoys the
second-order integrable quasi-Lipschitzian property \eqref{second_order_ Int_Lips_like_inq} around $(\ox,\ooy)$. Then for every $y^\ast\in\Y$ we have
\begin{equation}\label{equation_second_order}
\breve{\partial}^2\Intf{\ph}(\ox,\oy)(y^\ast)\subset\int_T{\partial}^2\ph_t\big(\ox,\ooy(t)\big)(y^\ast)d\mu.
\end{equation}
\end{theorem}\vspace*{-0.05in}
{\bf Proof}. Consider again the set-valued normal integrand $\Phi\colon T\times\X\tto\X$ given by $\Phi_t(x):=\partial\ph_t(x)$. Observe that by \eqref{cond_reg_sub0} and \eqref{cond_reg_sub} the
multifunction $\Phi$ satisfies the conditions in \eqref{convex_cond_set} and $\partial\Intf{\ph}=\Intfset{\Phi}$, respectively, around $\ox$.  Furthermore, it follows from
\eqref{second_order_ Int_Lips_like_inq} that $\Phi$ enjoys the integrable quasi-Lipschitzian property around $(\ox,\ooy)$. Applying Theorem~\ref{Theo_lim_reg_Cod}, we verify the claimed inclusion
\eqref{equation_second_order}. $\h$\vspace*{0.05in}

The next theorem is the main result of this section. It establishes the pointwise second-order subdifferential Leibniz rule for the robust {\em basic} second-order construction \eqref{2nd}.
\vspace*{-0.05in} \begin{theorem}[\bf basic second-order subdifferential Leibniz rule]\label{Theo_secordsub} In the setting of Theorem~{\rm\ref{2leib-comb}}, assume in addition that the mapping
${\cal S}_\ph$ from \eqref{Sph} is inner semicompact at $(\ox,\oy)$. Then we have \begin{equation}\label{2leib-bas} \partial^2\Intf{\ph}(\ox,\oy)(y^\ast)\subset\bigcup\limits_{\ooy\in
S_{\ph}(\ox,\oy)}\int_T\partial^2\ph_t\big(\ox,\ooy(t)\big)(y^\ast)d\mu. \end{equation} \end{theorem}\vspace*{-0.1in} {\bf Proof}. Having $\Phi\colon T\times\X\tto\X$ with
$\Phi_t(x):=\partial\ph_t(x)$, observe as in the proof of Theorem~\ref{2leib-comb} that $\Phi$ satisfies the conditions in \eqref{convex_cond_set} and $\partial\Intf{\ph}=\Intfset{\Phi}$ around
$\ox$, and that $\Phi$ enjoys the integrable quasi-Lipschitzian property around $(\ox,\ooy)$ for all $\ooy\in\mathcal{S}_\Phi(\ox,\oy)= \mathcal{S}_\ph(\ox,\oy)$. Since $\mathcal{S}\ph$ is inner
semicompact at $(\ox,\oy)$, we get that $\mathcal{S}_\Phi$ is also inner semicompact at this point. Applying Theorem~\ref{Theo_Basic_coder} tells us that \begin{align*}
\partial^2\Intf{\ph}(\ox,\oy)(y^\ast)&={D}^\ast\Intfset{\Phi}(\ox,\oy)(y^\ast)\subset\bigcup\limits_{\ooy\in\mathcal{S}_{\Phi}(\ox,\oy)}\int_T{D}^\ast\Phi_t\big(\ox,\ooy(t)\big)(y^\ast)d\mu\\&=
\bigcup\limits_{\ooy\in
S_{\ph}(\ox,\oy)}\int_T\partial^2\ph_t\big(\ox,\ooy(t)\big)(y^\ast)d\mu, \end{align*} which verifies \eqref{2leib-bas} and thus completes the proof of this theorem. $\h$\vspace*{0.05in}

The final result of this section provides a second-order subdifferential Leibniz rule for normal integrands in \eqref{eif}, which are represented in the form of {\em maximum functions}
\begin{equation}\label{max} \ph(t,x):=\max\big\{\psi_i(t,x)\;\big|\;i=1,\ldots,s\big\}\;\mbox{ for }\;t\in T\;\mbox{ and }\;x\in\X. \end{equation} The second-order Leibniz rule, which is obtained
in the next statement, evaluates the (basic) second-order subdifferential of $\Intf{\ph}$ with $\ph$ from \eqref{max} in terms of the second-order subdifferential of $\ph_t$. The latter
construction is constructively {\em calculated} in \cite{eh,ms} entirely via the given functions $\psi_i$ for various types of maximum functions. For brevity we do not present the precise
formulas here while referring the reader to the aforementioned papers. \begin{corollary}[\bf second-order subdifferential Leibniz rule with normal integrands as maximum functions]\label{2max} Let
$\psi_i\colon T\times\X\to\Rex$, $i=1,\ldots,s$, be a family of normal integrand functions on a complete finite measure space $(T,{\cal A},\mu)$, and let $U$ be an open subset of $\X$ where the
functions $\psi_i$ satisfy the following assumptions: \begin{enumerate}[label=\alph*)] \item[\bf(i)] When $x\in U$ and $i=1,\ldots,s$, the functions $\psi^i_t$ are continuously differentiable
around $x$ for a.e.\ $t\in T$. \item[\bf(ii)] There exists $\ell\in\Leb^1(T,\R_+)$ such that for all $i=1,\dots,s$, all $u_1,u_2\in U $, and a.e.\ $t\in T$ we have \begin{equation*}
|\psi_i(t,u_1)-\psi_i(t,u_2)|+\|\nabla_x\psi_i(t,u_1)-\nabla_x\psi_i(t,u_2)\|\le\ell(t)\|u_1-u_2\|. \end{equation*} \end{enumerate} Consider the maximum function \eqref{max} and assume that for
some $\ox\in U$ and $\oy\in\partial\Intf{\ph}(\ox)$ the multifunction $\mathcal{S}_f$ is inner semicompact at $(\ox,\oy)$ and that $\ph$ enjoys the second-order integrable quasi-Lipschitzian
property \eqref{second_order_ Int_Lips_like_inq} around $(\ox,\ooy)$. Then for all $y^\ast\in\X$ we have the second-order subdifferential Leibniz rule \eqref{2leib-bas}, where the second-order
subdifferentials $\partial^2\varphi_t$, $t\in T$, are computed in {\rm\cite{eh,ms}}. \end{corollary}\vspace*{-0.05in} {\bf Proof}. It is well known (see, e.g., \cite[Theorem~4.10]{m18}) that
\begin{equation}\label{sub:max}
 \partial\ph_t(x)={\rm co}\big\{\nabla_x\psi_i(t,x)\;\big|\;\psi_i(t,x)=\ph(t,x)\big\}\;\mbox{ for all }\;x\in U.
\end{equation} Thus $\ph_t$ satisfies the conditions \eqref{cond_reg_sub0} around $\ox$. It follows from Proposition~\ref{suf_condition_to_cond_reg_sub}(ii) that \eqref{cond_reg_sub} holds at
every $x\in U$. Consequently, we get the equality \begin{equation*} \partial\Intf{\ph}(x)=\Intf{\Phi}(x)\;\text{ for all }\;x\in U\;\mbox{ with }\;\Phi(t,u):=\partial\ph_t(u). \end{equation*}
Employing finally  Theorem~\ref{Prop47} and Proposition~\ref{eq:regular:basic} tells us that the maximum function $\ph$ satisfies the second-order integrable quasi-Lipschitzian assumption
required in Theorem~\ref{Theo_secordsub}. Thus applying Theorem~\ref{Theo_secordsub} we arrive at \eqref{2leib-bas} and complete the proof of the corollary. $\h$\vspace*{-0.1in}

\section{Concluding Remarks}\label{sec7}\sce\vspace*{-0.1in}

The paper establishes general results on first-order and second-order generalized differentiation of random set-valued and single-valued mappings. We also introduce new Lipschitzian properties of
such multifunctions and derive their generalized differential characterizations. The obtained results provide the foundation for broad applications of variational analysis and generalized
differentiation to stochastic optimization and related topics at the same level of perfection as for deterministic counterparts. Among our future research topics, we mention applications to
first-order and second-order optimality conditions in stochastic programming, sensitivity analysis in parametric stochastic optimization, tilt and full stability of random optimal solutions,
stochastic variational inequalities, and stochastic numerical algorithms. Classes of problems in stochastic optimization of our special interest include two-stage stochastic programs,
probabilistic programs, stochastic bilevel programs, etc. We are positive that the theory developed in this paper will be highly instrumental in applications to such classes of stochastic
problems. \\[1ex] {\bf Acknowledgements}. The authors thank the Handling Associate Editor and two anonymous referees for their useful suggestions and remarks, which helped us to improve the
original presentation. \end{document}